\documentclass[12pt]{article}
\usepackage{amssymb,amsmath,amsthm,amscd,latexsym}
\usepackage{mathrsfs}
\usepackage{indentfirst}
\usepackage[shortlabels]{enumitem}
\usepackage{anysize}
\usepackage{hyperref}
\usepackage{color}
\usepackage{longtable}

\renewcommand{\paragraph}{\roman{paragraph}}
\input cyracc.def

\newcommand{\Q}{\mathbb{Q}}

\newcommand{\Z}{\mathbb{Z}}
\newcommand{\BP}{{\mathbb P}}

\newcommand{\F}{\mathbb{F}}

\newcommand{\qbinom}[3]{\genfrac{[}{]}{0pt}{}{#1}{#2}_{#3}}

\DeclareMathOperator{\PG}{PG}
\DeclareMathOperator{\RM}{RM}
\DeclareMathOperator{\Sim}{Sim}

\newtheorem{theorem}{Theorem}[section]
\newtheorem{lemma}[theorem]{Lemma}
\newtheorem{corollary}[theorem]{Corollary}

\newtheorem{conjecture}[theorem]{Conjecture}
\numberwithin{table}{section}
\theoremstyle{definition}
\newtheorem{definition}[theorem]{Definition}
\theoremstyle{remark}
\newtheorem{remark}[theorem]{Remark}

\newcommand{\SWRGparam}{s}

\begin{document}
\title{\bf On strongly walk regular graphs,\\ triple sum sets and their codes}

\author{
Michael Kiermaier\thanks{Department of Mathematics, University of Bayreuth, Bayreuth, Germany},
Sascha Kurz$^*\!\!$, 
Patrick Sol\'e\thanks{I2M, CNRS, Aix-Marseille Univ, Centrale Marseille, Marseille, France},\\
Michael Stoll$^*$, and 
Alfred Wassermann$^*$
}

\date{}

\maketitle
\noindent
{\bf Abstract:} {
Strongly walk regular graphs (SWRGs or $s$-SWRGs) form a natural generalization of strongly regular graphs (SRGs) where paths of length~2 are replaced by paths of length~$s$.
They can be constructed as coset graphs of the duals of projective three-weight codes whose weights satisfy a certain equation.
We provide classifications of the feasible parameters of these codes in the binary and ternary case for medium size code lengths.
For the binary case, the divisibility of the weights of these codes is investigated and several general results are shown.

It is known that an $s$-SWRG has at most 4 distinct eigenvalues $k > \theta_1 > \theta_2 > \theta_3$, and that the triple $(\theta_1, \theta_2, \theta_3)$ satisfies a certain homogeneous polynomial equation of degree $s - 2$ (Van Dam, Omidi, 2013).
This equation defines a plane algebraic curve; we use methods from algorithmic arithmetic geometry to show that for $s = 5$ and $s = 7$, there are only the obvious solutions, and we conjecture this to remain true for all (odd) $s \ge 9$.}



\section{Introduction}
\label{sec_introduction}
A strongly regular graph (SRG) is a regular graph such that the number of common neighbors of two distinct vertices depends only on whether these vertices are
adjacent or not. They arise in a lot of applications, see e.g.\ \cite{brouwer2011spectra}. As first observed in \cite{delsarte1972weights}, there is a strong link
to projective two-weight codes, see \cite{calderbank1986geometry} for a survey. The notion of SRGs has been generalized to distance-regular graphs or association
schemes. Noting that the number of common neighbors of two vertices equals the number of walks of length two between them, strongly walk-regular graphs (SWRG) were
introduced in \cite{van2013strongly}. A graph is an $\SWRGparam$-SWRG if the number of walks of length $\SWRGparam$ from a vertex to another vertex depends only on
whether the two vertices are the same, adjacent, or not adjacent. Note that SRGs are $s$-SWRGs for all $\SWRGparam>1$.
In \cite[Theorem 3.4]{van2013strongly} it is shown that
the adjacency matrix of a SWRG has at most four distinct eigenvalues and the following characterization of SWRGs is given.

\begin{lemma}[van Dam, Omidi {\cite[Proposition~4.1]{van2013strongly}}]
\label{lem:vandam_omidi}
Let $\Gamma$ be a $k$-regular graph with four distinct eigenvalues $k>\theta_1>\theta_2>\theta_3$.
Then $\Gamma$ is an $\SWRGparam$-SWRG for $\SWRGparam\ge 3$ if and only if
\begin{equation}
  (\theta_2-\theta_3)\theta_1^{\SWRGparam}+(\theta_3-\theta_1)\theta_2^{\SWRGparam}+(\theta_1-\theta_2)\theta_3^{\SWRGparam}=0.\label{eq_diophantine_vandam}
\end{equation}
\end{lemma}
Moreover, it is known that $\SWRGparam$ has to be odd. All known examples for $\SWRGparam$-SWRGs with $\SWRGparam>3$ satisfy
$\theta_2=0$ and $\theta_3=-\theta_1$, where Equation~(\ref{eq_diophantine_vandam}) is automatically satisfied for all odd
$\SWRGparam\ge 3$.

Mimicking the mentioned link between SRGs and projective two-weight codes, a construction of SRWGs as coset graphs of the duals of projective three-weight codes was
given recently in \cite{shi2019three}.
The eigenvalues of such graphs are integral and depend on the weights of the three-weight code, so that Equation~(\ref{eq_diophantine_vandam})
turns into a number theory question.
In \cite{shi2019three}, a construction of SWRGs from triple sum sets (TSS) is given.

Several research papers consider the feasible parameters of SRGs, see e.g.~\cite{brouwerSRG} for a large table
together with references summarizing the state of knowledge.
We remark that the smallest cases, where the existence or non-existence of a SRG is unclear,
consist of $65$ or $69$ vertices. The corresponding
parameters cannot be attained by two-weight codes since these always give graphs
where the number of vertices is a power of the field size.
Still, the existence of projective two-weight codes is an important source for
the construction of SRGs,
see e.g.~\cite{kohnert2007constructing}, where several new examples have
been found. An online database of known two-weight codes can be found at
\cite{twoweighttables}. Due to a result
of Delsarte \cite[Corollary 2]{delsarte1972weights} the possible weights
of two-weight codes are quite restricted, see Lemma~\ref{lemma_delsarte}.

Given the relation between the weights of a projective three-weight code and the eigenvalues of the coset graph of its dual, corresponding solutions of
Equation~(\ref{eq_diophantine_vandam}) can be easily enumerated.
However, not all cases are feasible, i.e., attainable by a projective three-weight code.
The aim of this paper is to study feasibility for the smallest cases.
For binary codes we give results for lengths smaller than $72$ and for ternary codes for lengths
up to $39$. Within that range only very few cases are left as open problems.
This extends and corrects first enumeration results from \cite{shi2019three}.
Similar results for some special rings instead of finite fields are obtained in \cite{kiermaier2019three}.

The remaining part of this paper is organized as follows.
The necessary preliminaries are introduced in Section~\ref{sec_preliminaries} followed by the enumeration results
in Section~\ref{sec_feasible_parameters}. In Section~\ref{sec_sum_monomials} it is shown that for $s=5$ and $s=7$ the only rational solutions of Equation~(\ref{eq_diophantine_vandam})
are given by the parametric solution $\theta_2=0$, $\theta_3=-\theta_1$.
For $s = 5$, this reduces to the determination of the set of rational points on an elliptic curve and for $s = 7$, it leads to a curve of genus~$2$.

The computational results from Section~\ref{sec_feasible_parameters} for the case $q=2$ suggest
that projective three-weight codes of length $n$ whose weights satisfy
$w_1+w_2+w_3=3n/2$ possess a high divisibility of the weights and the length by powers of two.
In Section~\ref{sec_divisibility} this is shown, see Lemma~\ref{lem:divisibility_2} and the following theorems for the details.
In Appendix~\ref{appendix_generator_matrices} we collect generator matrices for the mentioned feasible parameters from
Section~\ref{sec_feasible_parameters}.

\section{Preliminaries}
\label{sec_preliminaries}

In this article, $q\geq 2$ will always be the power of some prime $p$.

A linear $q$-ary code $C$ of length $n$ and dimension $k$ is called an $[n,k]_q$ code.
The number of positions which are not all-zero is called the \emph{effective length} of $C$.
If the length equals the effective length, $C$ is called \emph{full-length}.
Two positions $i,j\in \{1,\ldots,n\}$ of $C$ are called \emph{projectively equivalent} if there is a $\lambda\in\F_q^*$ with $c_i = \lambda c_j$ for all codewords $c\in C$.
The code $C$ is called \emph{projective} if it is full-length and there are no projectively equivalent positions.
For a general full-length code, the \emph{position multiplicity type} is the sequence $(m_i)$ where $m_i$ denotes the number of projective equivalence classes of size $i$.

\subsection{Restrictions on the weights}

If there is only a single non-zero weight, $C$ is called a \emph{constant weight code}.
The constant weight codes are completely classified.

\begin{lemma}[Bonisoli \cite{Bonisoli1984}]\label{lemma_1wt}
Let $C$ be a full-length $[n,k]_q$ code of constant weight $w$ of dimension $k\geq 1$.
Then $q^{k-1} \mid w$, and $C$ is isomorphic to the $u$-fold repetition of the $q$-ary simplex code $\Sim_q(k)$ of dimension $k$ with $u = w / q^{k-1}$.
In particular, $n = u \cdot \frac{q^k - 1}{q-1}$.
\end{lemma}


If $C$ has exactly two different non-zero weights, $C$ is called a \emph{two-weight code}.

\begin{lemma}[Delsarte {\cite[Corollary 2]{delsarte1972weights}}]\label{lemma_delsarte}
  Let $C$ be a projective two-weight code over $\F_q$, where $q=p^e$ for some prime $p$.
  Then there exist suitable integers $u$ and $t$ with $u\ge 1$, $t\ge 0$ such that the weights are given by $w_1=up^t$ and $w_2=(u+1)p^t$.
\end{lemma}

If $C$ has exactly three different non-zero weights, $C$ is called a \emph{three-weight code}.
Furthermore, $C$ is called \emph{$\Delta$-divisible} for some integer $\Delta \geq 1$ if all weights of $C$ are divisible by $\Delta$.

\begin{lemma}
	\label{lem:3wt_length_restriction}
	Let $C$ be a linear projective $[n,k]_q$ three-weight code.
	Then $n \leq \frac{q^k-1}{q-1} - 2$.
\end{lemma}

\begin{proof}
	Let $G$ be a generator matrix of $C$.
	Since $C$ is projective, $G$ neither has a zero column, nor a pair of projectively equivalent columns.
	So each of the $\frac{q^k-1}{q-1}$ projective equivalence classes of non-zero vectors in $\F_q^k$ appears at most once as a column of $G$, showing that $n \leq \frac{q^k-1}{q-1}$.
	In the case $n = \frac{q^k-1}{q-1}$, $C$ is the simplex code $\Sim_q(k)$ of dimension $k$ over $\F_q$, which is a code of constant weight $q^{k-1}$.
	In the case $n = \frac{q^k-1}{q-1} - 1$, $C$ is the simplex code $\Sim_q(k)$ punctured in a single position, so $C$ has only the two weights $q^{k-1}$ and $q^{k-1}-1$.%
  \footnote{In fact, in this case $C$ is a MacDonald code.}
  This contradicts the assumption that $C$ is a three-weight code.
\end{proof}

\subsection{Weight enumerators and the MacWilliams identity}

The \emph{weight distribution} of $C$ is the sequence of numbers $(A_i)$ where $A_i$ denotes the number of codewords of weight $i$.
It can also be denoted as $(0^{A_0} 1^{A_1} 2^{A_2} \ldots)$, where entries with $A_i = 0$ may be omitted.
The weight distribution is often given in polynomial form as the \emph{(univariate) weight enumerator} $W_C(x) = \sum_i A_i x^i$ or the \emph{homogeneous weight enumerator} $W_C(x,y) = \sum_i A_i x^{n-i} y^i$.

The weight distribution of the dual code $C^\perp$ will be denoted by $(B_i)$.
We always have $A_0 = B_0 = 1$.
Furthermore, $B_0 = 0$ if and only if $C$ is full-length and $B_0 = B_1 = 0$ if and only $C$ is projective.
For the number $B_2$ of general full-length codes, the following statement can be checked.

\begin{lemma}
	\label{lem:col_mult_B2}
	Let $C$ be a full-length $q$-ary linear code of length $n$ and $(m_i)$ the position multiplicity type of $C$.
	Then
	\begin{align*}
		\sum_i i m_i & = n \quad\text{and} \\
		\sum_i (q-1) \binom{i}{2} m_i & = B_2
	\end{align*}
\end{lemma}

The weight distributions of $C$ and $C^\perp$ are related via the \emph{MacWilliams identities} \cite{MacWilliams-1963-BSTJ42[1]:79-94}
\begin{equation}
  \sum_{j=0}^{n-\nu} {{n-j}\choose\nu} A_j=q^{k-\nu}\cdot \sum_{j=0}^\nu {{n-j}\choose{n-\nu}}B_j\quad\text{for }0\le\nu\le n\text{.}
\end{equation}
or in homogeneous polynomial form as
\[
	W_{C^\perp}(x,y) = \frac{1}{\#C} \cdot W_C(x + (q-1) y, x - y)\text{.}
\]
In fact, the $B_i$ are uniquely determined by the $A_i$, as can be seen by the following variant of the MacWilliams identities.
Based on the \emph{$i$-th $q$-ary Krawtchouk polynomial}
\[
K_i(x) = \sum_{\nu=0}^i (-1)^\nu \binom{x}{\nu} \binom{n-x}{i-\nu} (q-1)^{i-\nu}\text{,}
\]
we have
\[
	B_i = \frac{1}{\#C} \cdot \sum_{j = 0}^n K_i(j) A_j\text{.}
\]
For a binary projective $[n,k]_2$ code, the system of the four equations with $i\in\{0,1,2,3\}$ can be rewritten to
\begin{eqnarray}
  \sum_{i>0} A_i &=& 2^k-1,\label{eq_ppm1proj}\\
  \sum_{i\ge 0} iA_i &=& 2^{k-1}n,\label{eq_ppm2proj}\\
  \sum_{i\ge 0} i^2A_i &=& 2^{k-2}\cdot n(n+1),\label{eq_ppm3proj}\\
  \sum_{i\ge 0} i^3A_i &=& 2^{k-3}\cdot (n^2(n+3)-6B_3)\label{eq_ppm4proj}.
\end{eqnarray}
In this special form of the left hand side, they are also called the first four \emph{(Pless) power moments}, see \cite{pless1963power}.
Given the length $n$, the dimension $k$, and the weights $w_1,w_2,w_3$ of a projective three-weight code, we can compute $A_{w_i}$ and
$B_3$:
\begin{eqnarray}
  A_{w_1} &=& \frac{2^{k-2}\cdot\left(n^2 - 2nw_2 - 2nw_3 + 4w_2w_3 + n\right) - w_2w_3}{(w_2-w_1)(w_3-w_1)} \label{eq_A1}\\
  A_{w_2} &=& \frac{2^{k-2}\cdot\left(n^2 - 2nw_1 - 2nw_3 + 4w_1w_3 + n\right) - w_1w_3}{(w_2-w_3)(w_2-w_1)} \label{eq_A2}\\
  A_{w_3} &=& \frac{2^{k-2}\cdot\left(n^2 - 2nw_1 - 2nw_2 + 4w_1w_2 + n\right) - w_1w_2}{(w_3-w_1)(w_3-w_2)} \label{eq_A3}\\
  3B_3 &=& \frac{n^2(n+3)}{2}-\left(w_1+w_2+w_3\right)n(n+1)  \nonumber\\
  && +2\left(w_1w_2+w_1w_3+w_2w_3\right)n-4w_1w_2w_3+w_1w_2w_3\cdot 2^{2-k}\label{eq_B3}
\end{eqnarray}
All $A_j$ except $A_0=1$ and $A_{w_1}$, $A_{w_2}$, $A_{w_3}$ are equal to zero, so that the $B_i$ with $i\ge 4$ are be uniquely determined using the remaining MacWilliams identities, i.e., those for $\nu\ge 4$.
Note that (\ref{eq_B3}) implies that the product $w_1w_2w_3$ has to be divisible by $2^{k-2}$.
We remark that we will obtain stronger divisibility conditions in Section~\ref{sec_divisibility}.
Of course, similar explicit expressions can also be determined for field sizes $q>2$.
However, we will mostly restrict our theoretical considerations to $q=2$ in the remaining part of the paper.

For a linear $[n_1,k_1]_q$ code $C_1$ and a linear $[n_2,k_2]_q$ code $C_2$, the \emph{direct sum} of $C_1$ and $C_2$ is defined as
\[
	C_1 \oplus C_2 = \{ (c_1, c_2) \mid c_1\in C_1, c_2\in C_2\}\text{.}
\]
It is a linear $[n_1 + n_2, k_1 + k_2]_q$ code.
Its weight enumerator is given by
\[
	W_{C_1 \oplus C_2}(x) = W_{C_1}(x) \cdot W_{C_2}(x)\text{.}
\]

\subsection{The coset graph triple sum sets}

A \emph{coset} of a linear code $C$ is any translate of $C$ by a constant vector. A \emph{coset leader} of a fixed coset is any element
that minimizes the weight. The \emph{weight} of a coset is the weight of any of its coset leaders. With this, the
\emph{coset graph} $\Gamma_C$ of a linear code $C$ is defined on the cosets of $C$ as vertices, where two cosets are connected
iff they differ by a coset of weight one. To ease notation, we speak of the eigenvalues of a graph $\Gamma$ meaning the
eigenvalues of the corresponding adjacency matrix. For a projective code $C$ the eigenvalues of the coset graph $\Gamma_{C^\perp}$
of its dual code are completely determined by the occurring non-zero weights $w_i$ of $C$, see \cite[Theorem 1.11.1]{brouwer1989distance}:
\begin{theorem}
\label{thm_eigenvalue_formula}
Let $C$ be a projective $[n,k]_q$ code with distinct weights $w_0=0,w_1$, $\dots$, $w_r$. Then, the coset graph $\Gamma_{C^\perp}$
of its dual code $C^\perp$ is $n(q-1)$-regular and the eigenvalues are given by $n(q-1)-qw_i$ for $i\in\{0,\ldots,r\}$.
\end{theorem}

\emph{Triple sum sets} (TSS) have been introduced in \cite{courteau1984triple} as generalization of partial difference sets.
A set
$\Omega\subseteq \mathbb{F}_q^k$ is called a \emph{triple sum set} if it is closed under scalar multiplication and there are constants
$\sigma_0$ and $\sigma_1$ such that each non-zero $h\in \mathbb{F}_q^k$ can be written as $h=x+y+z$ with $x,y,z\in \Omega$
exactly $\sigma_0$ times if $h\in\Omega$ and $\sigma_1$ times if $h\in\mathbb{F}_q^k\backslash\Omega$.

If $\Omega\subseteq\mathbb{F}_q^k$ and $0\notin\Omega$, then we denote by $C(\Omega)$ the projective code of length
$n=\#\Omega/(q-1)$ obtained as the kernel of the $k\times n$ matrix $H$ whose columns are the projectively non-equivalent elements
of $\Omega$. Thus, $H$ is the parity check matrix of the linear code $C(\Omega)$. In order to ease the notation, we abbreviate
$\Gamma_{C(\Omega)}$ as $\Gamma(\Omega)$. In \cite[Theorem 2]{shi2019three} it was shown that $\Omega$ is a TSS if and only if
$\Gamma(\Omega)$ is a $3$-SWRG. (Actually, \cite[Theorem 2]{shi2019three} states the equivalence of $\Gamma(\Omega)$ being an
$s$-SWRG and $\Omega$ being an $s$-sum set, where the element $h$ in the definition of a TSS is a sum of $s$ elements from $\Omega$.)

\begin{lemma}\label{lem:reltheta}
Let $s\geq 2$ be an integer.
The following equation holds for all  $\theta_1$, $\theta_2$, $\theta_3$ over any commutative ring:
\begin{equation} \label{E:reltheta}
  (\theta_2-\theta_3)\theta_1^{\SWRGparam}+(\theta_3-\theta_1)\theta_2^{\SWRGparam}+(\theta_1-\theta_2)\theta_3^{\SWRGparam}
  =(\theta_1-\theta_2)(\theta_1-\theta_3)(\theta_2-\theta_3)\cdot\!\!\!\!\sum_{h+i+j=\SWRGparam-2} \theta_1^h\theta_2^i\theta_3^j\,.
\end{equation}
\end{lemma}

A coding-theoretic characterization of triple sum sets is given as follows, see \cite[Theorem 2.1]{courteau1984triple} or
\cite[Theorem 5]{shi2019three}.
\begin{theorem}
  If $\Omega\subseteq\mathbb{F}_q^k$ so that $C(\Omega)^\perp$ has length $n$ and attains exactly three non-zero weights
  $w_1$, $w_2$, and $w_3$, then $\Omega$ is a TSS iff $w_1+w_2+w_3=\frac{3n(q-1)}{q}$.
\end{theorem}
\begin{proof}
%
Using Equation~(\ref{E:reltheta}) from Lemma~\ref{lem:reltheta} for $\SWRGparam=3$, (\ref{eq_diophantine_vandam}) becomes
\[
(\theta_1-\theta_2)(\theta_1-\theta_3)(\theta_2-\theta_3)(\theta_1 + \theta_2 + \theta_3) = 0\,.
\]
Theorem~\ref{thm_eigenvalue_formula} shows that the eigenvalues are pairwise different and
we conclude that (\ref{eq_diophantine_vandam})
is satisfied iff $\theta_1+\theta_2+\theta_3=0$.
Plugging in the
formula for the eigenvalues from Theorem~\ref{thm_eigenvalue_formula} gives the condition $w_1+w_2+w_3=\frac{3n(q-1)}{q}$.
\end{proof}
As mentioned in the introduction, all known examples for $\SWRGparam$-SWRGs satisfy $\theta_2=0$ and $\theta_3=-\theta_1$, i.e., they
are $s$-SWRGs for all odd $\SWRGparam\ge 3$. So, starting from projective three-weight codes to construct SWRGs it is
sufficient to study those that satisfy the weight constraint $w_1+w_2+w_3=3n(q-1)/q$.
We do so in Section~\ref{sec_feasible_parameters}.
We would like to point out that there are (many) binary projective three-weight codes with $w_1 + w_2 + w_3 \neq 3n(q-1)/q$.
As an example, consider the binary $[6,5]_2$ parity check code.
It is a projective three-weight-code with weight distribution is $(0^1 2^{15} 4^{15} 6^1)$.
The sum of its weights is $12$, but $\frac{3}{2} n = 9$.

\subsection{The geometric point of view to linear codes}
For a vector space $V$, let $\qbinom{V}{k}{q}$ be the set of all subspaces of dimension $k$.
If $V$ is of finite dimension $v$, $\#\qbinom{V}{k}{q}$ equals the Gaussian binomial coefficient $\qbinom{v}{k}{q}$.
Conveniently, we will identify a vector space $V$ with the set $\qbinom{V}{1}{q}$ of \emph{points} contained in $V$.

A \emph{multiset} on a base set $S$ is a mapping $M : S \to \mathbb{Z}_{\geq 0}$, assigning a multiplicity to each element in $S$.
For $T \subseteq S$, we set $M(T) = \sum_{s\in T}M(s)$.
The \emph{cardinality} of $M$ is $\#M = M(S)$, which is the sum of the multiplicities of all elements.
In enumerative form, a multiset may be written by statements of the form $M = \{\!\!\{ s_1, s_2, \ldots, s_n\}\!\!\}$, with the obvious interpretation.
If $N$ is a multiset on a base set $T$ and $\phi$ is a predicate on $T$, we may also use the multiset-builder notations like $M = \{\!\!\{t\in N \mid \phi(t)\}\!\!\}$.

We will make use of the geometric description of linear codes as in \cite{Dodunekov-Simonis-1998-ElecJComb5:R37}.
There is a bijective correspondence of (semi-)linear equivalence classes of linear full-length $[n,k]_q$-codes $C$ and (semi-)linear equivalence classes of spanning multisets $\mathcal{C}$ of $n$ points in $\PG(V) \cong \PG(k-1,q)$, where $V$ is a $\F_q$-vector space of dimension $k$.
For a concrete assignment, let $G$ be a generator matrix of $C$ and $v_1,\ldots,v_n$ the columns of $G$ and consider the multiset $\mathcal{C} = \{\!\!\{\langle v_1\rangle,\ldots,\langle v_n\rangle\}\!\!\}$ of points in $V = \F_q^k$.
In this way, a codeword $c = xG$ is represented by the hyperplane $H = x^\perp$ (in fact, $H$ represents all the $q-1$ codewords which are projectively equivalent to $c$).
The weight of $c$ has a natural geometric description, namely $w(c) = n - \mathcal{C}(H)$, where the hyperplane $H$ is identified with the set of points contained in $H$.
In other words, $w(c)$ is the number of points in $\mathcal{C}$, counted with multiplicity, which are not contained in $H$.
The code $C$ is projective if and only if $\mathcal{C}$ is a proper set.

The following codes have an easy geometric description:
\begin{enumerate}[(i)]
	\item The $q$-ary simplex code $\Sim_q(k)$ of dimension $k\geq 1$ corresponds to the set of all points contained in a vector space of (algebraic) dimension $k$.
	It is a projective linear $[(q^k-1)/(q-1), k]_q$ constant code weight code with weight enumerator $W_{\Sim_q(k)}(x) = 1 + (q^k-1)x^{q^{k-1}}$.
	\item The $q$-ary first order Reed-Muller code $\RM_q(k)$ of dimension $k\geq 2$ corresponds to an affine subspace of dimension $k-1$, that is the set of points contained in $A = V \setminus W$, where $V$ is an $\F_q$-vector space of dimension $k$ and $W$ is a subspace of codimension $1$.
	It is a projective linear $[q^{k-1}, k]_q$ two-weight code with weight enumerator $W_{\RM_q(k)}(x) = 1 + (q^k-q)x^{(q-1)q^{k-2}} + (q-1) x^{q^{k-1}}$.
	In the geometric description, the space $W$ is known as the \emph{hyperplane at infinity} of $A$.
	It corresponds to the $q-1$ codewords of weight $q^{k-1}$.
\end{enumerate}

For a fixed point $P$ in $V$, we consider the standard projection $\pi_P : V \to V/P, x \mapsto x + P$.
It is extended to the multiset $\mathcal{C}$ of points as
\[
	\pi_P(\mathcal{C}) = \{\!\!\{ \pi_P(Q) \mid Q\in \mathcal{C} \text{ with }Q \neq P\}\!\!\}\text{.}
\]
We have $\#\pi_P(\mathcal{C}) = \#\mathcal{C} - \mathcal{C}(P)$.

The projections $\pi_P(\mathcal{C})$ with $P\in\qbinom{V}{1}{q}$ correspond to the subcodes $C'$ of $C$ of codimension one.
A codeword $c\in C$ with corresponding hyperplane $H < V$ is contained in $C'$ if and only if $P \in H$.

Let $C$ be a projective $[n,k]_q$-code with weight enumerator $W_C(x) = \sum_i A_i x^i$ and $\mathcal{C}$ a corresponding point set.
If the complement $\mathcal{C}^\complement  = \qbinom{\F_q^k}{1}{q} \setminus \mathcal{C}$ is spanning (or equivalently, $A_{q^{k-1}} = 0$), we call its corresponding $[\frac{q^k-1}{q-1}-n,k]_q$-code $C^\complement$ the \emph{anticode} of $C$.
In this way, the anticode of $C$ is defined up to isomorphism.
Its weight enumerator is $W_{C^\complement}(x) = 1 + \sum_{i > 0} A_{q^{k-1} - i} x^i$.

\section{Feasible parameters of projective three-weight codes satisfying $\mathbf{w_1+w_2+w_3=3n(q-1)/q}$}
\label{sec_feasible_parameters}
As outlined in Section~\ref{sec_preliminaries} we can construct $3$-SWRGs from projective $[n,k]_q$ three-weight codes
if the weights satisfy $w_1+w_2+w_3=3n(q-1)/q$.
So, here we study the feasible sets of parameters $n,k,w_1,w_2,w_3$ such that a corresponding projective three-weight
code exists. In Subsection~\ref{subsec_feasible_q_2} we consider the admissible parameters for all lengths $n<72$ in the binary case and in
Subsection~\ref{subsec_feasible_q_3} we consider the admissible parameters for all lengths $n\le 39$ in the ternary case.

In that range we can simply loop over all weight-triples $(w_1, w_2, w_3)$ with $1 \leq w_1 < w_2 < w_3 \leq n$ satisfying $w_1+w_2+w_3=3n(q-1)/q$.
For $q=2$, (\ref{eq_B3}) implies that the product $w_1w_2w_3$ is divisible by $2^{k-2}$,
which restricts the possible choices for the dimension $k$.
For $q=3$ we may use the trivial bounds $1\le k\le n$.
Then, the MacWilliams identities uniquely determine the values of all $A_i$s and $B_i$s.
As a first check we test if all of these values are non-negative integers.
As a consequence of \cite[Theorem 1]{ward1981divisible}, any full-length $\Delta$-divisible $[n,k]_q$ code is the $\Delta/\gcd(\Delta,q^{k-1})$-fold repetition of some code.
As projectivity forbids proper repetitions, we can restrict ourselves to the cases where $\gcd(\Delta,q^{k-1}) = \gcd(w_1, w_2, w_3, q)$ is a power of $p$.
Examples where we can apply this criterion to exclude the existence of codes are $q=2$, $n=36$,
$(w_1,w_2,w_3)=(12,18,24)$, and $k\in\{6,7,8\}$.
The corresponding values of $(A_{w_1}, A_{w_2},A_{w_3})$ are $(2,56,5)$, $(10,104,13)$, and $(26,200,29)$.
For $q=3$, this criterion can be applied to the parameters $n=24$, $k=4$ and weight triple $w=(14,16,18)$ as well as $n=36$, $k\in\{5,6\}$ and
weight triple $w=(18, 24, 30)$. In order to find examples, we have used the software package \texttt{LinCode} \cite{bouyukliev2021computer}
to enumerate matching codes or tried to reduce the problem complexity by prescribing automorphisms and applying exact or heuristic
solvers for the resulting integer linear programs.

Summarizing the above, we call parameters $(q, n, k, w_1, w_2, w_3)$ \emph{admissible}
if
\begin{enumerate}[(i)]
\item $1\le w_1<w_2<w_3\le n$ and $w_1+w_2+w_3=3n(q-1)/q$ and
\item $\gcd(w_1,w_2,w_3,q)$ is a power of $p$ and
\item $w_1w_2w_3$ is divisible by $2^{k-2}$ (if $q=2$) or $1\le k\le n$ (if $q=3$) and
\item all $A_i$ and $B_i$ with $i\in\{0,\ldots,n\}$ are non-negative integers and
\item $B_1 = B_2 = 0$.
\end{enumerate}

\subsection{Feasible parameters for projective binary three-weight codes with $\mathbf{w_1+w_2+w_3=3n/2}$}
\label{subsec_feasible_q_2}

In Table~\ref{table:q2} we list the admissible parameters for projective binary three-weight codes with $w_1+w_2+w_3=3n/2$.
For each length $4\leq n<72$ we list the admissible dimensions $k$, weight triples $w=(w_1,w_2,w_3)$, and
the weight distribution in the form $(A_{w_1},A_{w_2},A_{w_3})$.
The last column contains known results about the existence of codes with these parameters.
For some cases we can state the number of isomorphism types of those codes. The $8$-divisible $[n,k]_2$ codes with length at most $48$ are classified in
\cite{triply-Munemasa} and the projective codes are extracted in \cite{ubt_eref40887}.
If not mentioned otherwise, the remaining complete classification results are obtained with the
software package \texttt{LinCode} \cite{bouyukliev2021computer}.
For the parameters marked with $\ge 1$ we constructed at least one code by prescribing an automorphism group, see~\cite{bkw2005}.

We mark the non-existence results with the keyword {\lq\lq}\textbf{None}{\rq\rq} in the comment column of Table~\ref{table:q2} and give a reference to the used method.
One frequently showing up is the following.

\begin{lemma}(\cite[Proposition 5]{dodunekov1999some}, cf.~\cite{simonis1994restrictions})
  \label{lem:div_one_more}
  Let $C$ be an $[n,k,d]_2$-code with all weights divisible by $\Delta =2^a$ and let $\left(A_i\right)_{i=0,1,\dots,n}$ be the weight distribution of $C$. Put
  \begin{eqnarray*}
    \alpha&:=&\min\{k-a-1,a+1\},\\
    \beta&:=&\lfloor(k-a+1)/2\rfloor,\text{ and}\\
    \delta&:=&\min\{2\Delta i\,\mid\,A_{2\Delta i }\neq 0\text{ and } i>0\}.
  \end{eqnarray*}
  Then the integer
  $$
    T:=\sum_{i=0}^{\lfloor n/(2\Delta)\rfloor} A_{2\Delta i}
  $$
  satisfies the following conditions.
  \begin{enumerate}[label=(\roman*)]
    \item \label{div_one_more_case1}
          $T$ is divisible by $2^{\lfloor(k-1)/(a+1)\rfloor}$.
    \item \label{div_one_more_case2}
          If $T<2^{k-a}$, then
          $$
            T=2^{k-a}-2^{k-a-t}
          $$
          for some integer $t$ satisfying $1\le t\le \max\{\alpha,\beta\}$. Moreover, if $t>\beta$, then $C$ has an $[n,k-a-2,\delta]_2$-subcode and if $t\le \beta$, it has an
          $[n,k-a-t,\delta]_2$-subcode.
    \item \label{div_one_more_case3}
          If $T>2^k-2^{k-a}$, then
          $$
            T=2^k-2^{k-a}+2^{k-a-t}
          $$
          for some integer $t$ satisfying $0\le t\le \max\{\alpha,\beta\}$. Moreover, if $a=1$, then $C$ has an $[n,k-t,\delta]_2$ subcode.
          If $a>1$, then $C$ has an
          $[n,k-1,\delta]_2$ subcode unless $t=a+1\le k-a-1$, in which case it has an $[n,k-2,\delta]_2$ subcode.
  \end{enumerate}
\end{lemma}
A special and well-known subcase of Lemma~\ref{lem:div_one_more} is that the number of even weight codewords in a $[n,k]_2$ code is either
$2^{k-1}$ or $2^k$, see Lemma~\ref{lemma_even_weight_subcode}.
As an example, For $n=32$, $k=10$, and weight triple $w=(8, 16, 24)$ we obtain $(A_{w_1},A_{w_2},A_{w_3}) = (61,899,63)$. 
Applying Lemma~\ref{lem:div_one_more} gives $\Delta=8$, $a=3$, $\alpha=4$, $\beta=4$, $\delta=16$, and $T=900$.
As required by Part~\ref{div_one_more_case1}, $T$ is divisible by $4$.
However, Part~\ref{div_one_more_case3} gives $t=5$,
which contradicts $0\le t\le \max\{\alpha,\beta\}$, so that a code cannot exist.

Bounds for the largest possible minimum distance for given length and dimension are well studied in the literature, see e.g.\ the online tables \cite{grassl2007bounds}.
For length $n=64$ and dimension $k=11$ the largest possible minimum distance is known to be either $26$ or $27$,
which rules out the existence of a projective code with
weight triple $w=(28, 32, 36)$. We use the comment {\lq\lq}codetables{\rq\rq} in this case.
For $n=64$ and $w=(24, 32, 40)$ we use a classification result from
\cite{jaffe1997sextic}, i.e., every $13$-dimensional $8$-divisible binary linear code with non-zero weights in
$\{24,32,40,56,64\}$ has to contain a codeword of weight $64$.
Anticipating the results from Section~\ref{sec_divisibility} we also apply Corollary~\ref{cor_n_divisible_by_4},
which shows that the length $n$ has to be divisible
by $4$. The case $n=58$ is excluded by that criterion.
For length $n=64$ and weight triple $w=(16,32,48)$, the dimension can be at most $11$ by Theorem~\ref{thm:4Delta}.
Just four cases remain undecided. They occur for length $n\in\{40,48,56,64\}$ and
are marked by {\lq\lq}\textbf{Open}{\rq\rq}.
For each feasible case we give a suitable generator matrix as an example in Appendix~\ref{appendix_generator_matrices}.

  \begin{center}
    \begin{longtable}[c]{lllll}
    \caption{Admissible and realizable parameters of binary projective three-weight codes}\label{table:q2}\\
    $n$ & $k$ & $(w_1,w_2,w_3)$ & $(A_{w_1},A_{w_2},A_{w_3})$ & isom. types \\ \hline
    4 & 3  & $(1, 2, 3)$ & $(1,3,3)$ & 1 \\
    8 & $4$ & $(2, 4, 6)$ & $(1,11,3)$ & 1 \\
    8 & $5$ & $(2, 4, 6)$ & $(5,19,7)$ & 1 \\
    8 & $6$ & $(2, 4, 6)$ & $(13,35,15)$ & 1 \\
    12 & $5$ & $(4, 6, 8)$ & $(6,16,9)$ & 4 \\
    12 & $6$ & $(4, 6, 8)$ & $(18,24,21)$ & 2 \\
    16 & $5$ & $(6, 8, 10)$& $(6,15,10)$ & 5 \\
    16 & $6$ & $(6, 8, 10)$& $(22,15,26)$ & 1 \\
    16 & $7$ & $(6, 8, 10)$& $(54,15,58)$ & \textbf{None} Lem.~\ref{lem:div_one_more}\\
    16 & $5$ & $(4, 8, 12)$& $(1,27,3)$ & 1 \\
    16 & $6$ & $(4, 8, 12)$& $(5,51,7)$ & 1 \\
    16 & $7$ & $(4, 8, 12)$& $(13,99,15)$ & 2 \\
    20 & $5$ & $(8, 10, 12)$& $(5,16,10)$ & 3 \\
    20 & $6$ & $(8, 10, 12)$& $(25,8,30)$ & \textbf{None} Lem.~\ref{lem:div_one_more}\\
    24 & $5$ & $(10, 12, 14)$& $(3,19,9)$ & 1 \\
    24 & $6$ & $(10, 12, 14)$& $(27,3,33)$ & \textbf{None} Lem.~\ref{lem:div_one_more}\\
    24 & $6$ & $(8, 12, 16)$& $(6,48,9)$ & 8 \\
    24 & $7$ & $(8, 12, 16)$& $(18,88,21)$ & 52 \\
    24 & $8$ & $(8, 12, 16)$& $(42,168,45)$ & 66 \\
    24 & $9$ & $(8, 12, 16)$& $(90,328,93)$ & 13 \\
    24 & $10$ & $(8, 12, 16)$& $(186,648,189)$ & 2 \\
    24 & $11$ & $(8, 12, 16)$& $(378,1288,381)$ & 1 \\
    32 & $6$ & $(12, 16, 20)$& $(6,47,10)$ & $\ge$ 1\\
    32 & $7$ & $(12, 16, 20)$& $(22,79,26)$ & $\ge$ 1\\
    32 & $8$ & $(12, 16, 20)$& $(54,143,58)$ & $\ge$ 1\\
    32 & $9$ & $(12, 16, 20)$& $(118,271,122)$ & $\ge$ 1\\
    32 & $10$ & $(12, 16, 20)$& $(246,527,250)$ & $\ge$ 1\\
    32 & $6$ & $(8, 16, 24)$& $(1,59,3)$ & 1 \\
    32 & $7$ & $(8, 16, 24)$& $(5,115,7)$ & 1 \\
    32 & $8$ & $(8, 16, 24)$& $(13,227,15)$ & 2 \\
    32 & $9$ & $(8, 16, 24)$& $(29,451,31)$ & 1 \\
    32 & $10$ & $(8, 16, 24)$& $(61,899,63)$ & \textbf{None} Lem.~\ref{lem:div_one_more}\\  
    40 & $6$ & $(18, 20, 22)$& $(25,3,35)$ & \textbf{None} Lem.~\ref{lem:div_one_more}\\
    40 & $6$ & $(16, 20, 24)$& $(5,48,10)$ & $\ge$ 1\\
    40 & $7$ & $(16, 20, 24)$& $(25,72,30)$ & $\ge$ 1\\
    40 & $8$ & $(16, 20, 24)$& $(65,120,70)$ & $\ge$ 1\\
    40 & $9$ & $(16, 20, 24)$& $(145,216,150)$ & $\ge$ 1\\
    40 & $10$ & $(16, 20, 24)$& $(305,408,310)$ & \textbf{Open} \\
    48 & $6$ & $(22, 24, 26)$& $(18,15,30)$ & 1 \\
    48 & $6$ & $(20, 24, 28)$& $(3,51,9)$ & 1 \\
    48 & $7$ & $(20, 24, 28)$& $(27,67,33)$ & $\ge$ 209\,586 \\
    48 & $8$ & $(20, 24, 28)$& $(75,99,81)$ & $\ge$ 86 \\
    48 & $9$ & $(20, 24, 28)$& $(171,163,177)$ & \textbf{Open}\\
    48 & $7$ & $(16, 24, 32)$& $(6,112,9)$ &8 \\
    48 & $8$ & $(16, 24, 32)$& $(18,216,21)$ & 66 \\
    48 & $9$ & $(16, 24, 32)$& $(42,424,45)$ & $\ge $ 7 \\
    48 & $10$ & $(16, 24, 32)$& $(90,840,93)$ & $\ge 2$ \\
    48 & $11$ & $(16, 24, 32)$& $(186,1672,189)$  & $\ge$ 2 \\
    48 & $12$ & $(16, 24, 32)$& $(378,3336,381)$ & \\
    52 & $6$ & $(24, 26, 28)$& $(13,24,26)$ & 1 \\
    56 & $6$ & $(26, 28, 30)$& $(7,35,21)$ & 1 \\
    56 & $7$ & $(24, 28, 32)$& $(28,64,35)$ & $\ge$ 1\\
    56 & $8$ & $(24, 28, 32)$& $(84,80,91)$ & $\ge$ 1\\
    56 & $9$ & $(24, 28, 32)$& $(196,112,203)$ & $\ge$ 1\\
    56 & $10$ & $(24, 28, 32)$& $(420,176,427)$ & \textbf{Open}\\
    58 & $8$ & $(24, 31, 32)$& $(76,128,51)$ & \textbf{None} Cor.~\ref{cor_n_divisible_by_4} \\
    64 & $7$ & $(28, 32, 36)$& $(28,63,36)$ & $\ge$ 1\\
    64 & $8$ & $(28, 32, 36)$& $(92,63,100)$ & $\ge$ 1\\
    64 & $9$ & $(28, 32, 36)$& $(220,63,228)$ & $\ge$ 1\\
    64 & $10$ & $(28, 32, 36)$& $(476,63,484)$ & \textbf{Open} \\ 
    64 & $11$ & $(28, 32, 36)$& $(988,63,996)$ & \textbf{None} codetables \\
    64 & $7$ & $(24, 32, 40)$& $(6,111,10)$ & $\ge$ 1\\
    64 & $8$ & $(24, 32, 40)$& $(22,207,26)$ & $\ge$ 1\\
    64 & $9$ & $(24, 32, 40)$& $(54,399,58)$ & $\ge$ 1\\
    64 & $10$ & $(24, 32, 40)$& $(118,783,122)$ & $\ge$ 1\\
    64 & $11$ & $(24, 32, 40)$& $(246,1551,250)$ & 42 \\  
    64 & $12$ & $(24, 32, 40)$& $(502,3087,506)$ & 1 \\ 
    64 & $13$ & $(24, 32, 40)$& $(1014,6159,1018)$ & \textbf{None} \cite{jaffe1997sextic} \\ 
    64 & $7$ & $(16, 32, 48)$& $(1,123,3)$ & $\ge$ 1\\
    64 & $8$ & $(16, 32, 48)$& $(5,243,7)$ & $\ge$ 1\\
    64 & $9$ & $(16, 32, 48)$& $(13,483,15)$ & $\ge$ 1\\
    64 & $10$ & $(16, 32, 48)$& $(29,963,31)$ & $\ge$ 1\\
    64 & $11$ & $(16, 32, 48)$& $(61,1923,63)$ & 1 \cite{kurz2020classification}\\ 
    64 & $12$ & $(16, 32, 48)$& $(125,3843,127)$ & \textbf{None} Theorem~\ref{thm:4Delta}\\
    64 & $13$ & $(16, 32, 48)$& $(253,7683,255)$ & \textbf{None} Theorem~\ref{thm:4Delta}\\
    64 & $14$ & $(16, 32, 48)$& $(509,15363,511)$ & \textbf{None} Theorem~\ref{thm:4Delta}\\
    64 & $15$ & $(16, 32, 48)$& $(1021,30723,1023)$ & \textbf{None} Theorem~\ref{thm:4Delta}\\
    68 & $9$ & $(30, 32, 40)$& $(64,299,148)$ & \textbf{None} Theorem~\ref{thm:4Delta}\\
    \end{longtable}
  \end{center}

Based on \cite[Thm.~4]{Liu-2010-IntJInfCodingTheory1[4]:355-370} (for the projective case an alternative proof is found in \cite[Sec.~4]{Kiermaier-Kurz-2021-arXiv:2011.05872}), we derive the following classification result on three-weight codes.

\begin{theorem}
\label{thm:4Delta}
Let $\Delta = 2^a$ with $a\geq 3$ an integer and let $C$ be a full-length $[n,k]_2$ three-weight code with the non-zero weights $\Delta$, $2\Delta$ and $3\Delta$ and length $3\Delta \leq n \leq 4\Delta$.
Then $k \leq 2a+3$.
In the case of equality, we have that $n \in \{4\Delta - 1, 4\Delta\}$, $C$ is projective and falls into one of the following two cases.
\begin{enumerate}[(i)]
\item\label{thm:4Delta:minus1}
For $n = 4\Delta - 1$, $C$ is isomorphic to the direct sum of the binary simplex code of dimension $a+1$ and the binary first order Reed-Muller code of dimension $a+2$.
The weight enumerator of $C$ is
\[
	W_C(x) = 1 + (6\Delta - 3) x^{\Delta} + (8\Delta^2 - 8 \Delta + 3) x^{2\Delta} + (2\Delta - 1) x^{3\Delta}\text{.}
\]
\item\label{thm:4Delta:main}
For $n = 4\Delta$, $C$ is isomorphic to the code with the generator matrix
\[
\left(
\begin{array}{ccc|ccc}
\cline{1-3}
\multicolumn{1}{|c}{}  &         &   &   &         &                        \\
\multicolumn{1}{|c}{}  & R_{a+2} &   &   &         &                        \\ \cline{4-6}
\multicolumn{1}{|c}{1} & \cdots  & 1 &   &         & \multicolumn{1}{c|}{}  \\ \cline{1-3}
                       &         &   &   & R_{a+2} & \multicolumn{1}{c|}{}  \\
                       &         &   & 1 & \cdots  & \multicolumn{1}{c|}{1} \\ \cline{4-6}
\end{array}
\right) \in \F_2^{(2a+3)\times 4\Delta}\text{,}
\]
where
\[
	\begin{pmatrix}
	 & & \\
	 & R_{a+2} & \\
	 1 & \cdots & 1
	 \end{pmatrix}
	 \in\F_2^{(a+2) \times 2\Delta}
\]
denotes a generator matrix of the binary first order Reed-Muller code of dimension $a+2$, such that the all-one word is the last row of the generator matrix.
The weight enumerator of $C$ is
\[
	W_C(x) = 1 + (4\Delta - 3) x^{\Delta} + (8\Delta^2 - 8 \Delta + 3) x^{2\Delta} + (4\Delta - 1) x^{3\Delta}\text{.}
\]
\end{enumerate}

\end{theorem}

\begin{proof}
After appending zero positions, we may consider $C$ as a code of length $4\Delta$.
Let $\mathbf{1}$ be the all-one word of length $4\Delta$.
The code $\bar{C} = C + \langle \mathbf{1} \rangle$ is a $\Delta$-divisible binary linear code of effective length $4\Delta$ containing the all-one word $\mathbf{1}$.
By \cite[Thm.~4]{Liu-2010-IntJInfCodingTheory1[4]:355-370}, $\dim(\bar{C}) \leq 2a + 4$, and in the case of equality we may assume $\bar{C} = \RM_2(a+2) \oplus \RM_2(a+2)$, up to isomorphism.
So $k = \dim(C) \leq 2a + 3$, and in the case of equality, $C$ is a codimension $1$ subcode of $\bar{C}$ not containing $\mathbf{1}$.

We switch to the geometric description of linear codes.
The corresponding point set of $\bar{C} = \RM_2(a+2) \oplus \RM_2(a+2)$ has the form $\bar{\mathcal{C}} = A_1 \cup A_2$ with $A_1 = V_1 \setminus W_1$ and $A_2 = V_2 \setminus W_2$, where $V_1$ and $V_2$ are vector spaces over $\F_2$ of dimension $a+2$ having trivial intersection, and $W_1 < V_1$, $W_2 < V_2$ are codimension $1$ subspaces.%
\footnote{Remember that $V_1 \setminus W_1$ is a lazy way for writing $\qbinom{V_1}{1}{q} \setminus \qbinom{W_1}{1}{q}$.}
The ambient vector space is $V = V_1 \oplus V_2$.
The codeword $\mathbf{1}\in \bar{C}$ corresponds to a hyperplane $H_0$ of $V$ not containing any point of $\bar{\mathcal{C}}$.
By the dimension formula, $\dim(H_0 \cap V_1) \geq a+1$, which forces $H_0 \cap V_1 = W_1$.
In the same way, $H_0 \cap V_2 = W_2$ and therefore, $W_1 + W_2 < H_0$.
Since $W_1 + W_2$ has codimension $2$ in $V$, there are only $\qbinom{2}{1}{2} = 3$ hyperplanes of $V$ containing $W_1 + W_2$.
Two of these are $V_1 + W_2$ and $W_1 + V_2$ which do contain points of $\bar{\mathcal{C}}$, so $H_0$ is the third one.

As $C$ is a subcode of $\bar{C}$ of codimension $1$, a corresponding point set $\mathcal{C}$ of $C$ is given by the multiset image $\pi_P(\bar{\mathcal{C}})$ of the projection $\pi_P : V \to V/P$, $x\mapsto x + P$ with respect to a suitable point $P\in\qbinom{V}{1}{q}$.
Since $\mathbf{1} \notin C$, we have that $P \notin H_0$, so $P$ must be contained in one of the other two hyperplanes containing $W_1 + W_2$.
Without restriction, we may assume $P\in V_1 + W_2$.
Together with $P\notin H_0$, this implies $P\in (V_1 + W_2) \setminus (W_1 + W_2)$.

Case 1: $P\in \bar{\mathcal{C}}$, so $P\in V_1 \setminus W_1$.
We get that $\pi_P(A_1)$ is the set of all points in a subspace of algebraic dimension $a+1$, $\pi_P(A_2)$ is again an affine subspace of dimension $a+2$, and $\langle\pi_P(A_1)\rangle \cap \langle\pi_P(A_2)\rangle = \{\mathbf{0}\}$.
Therefore, $C \cong \Sim_2(a+1) \oplus \RM_2(a+2)$.
The weight enumerator is computed as $W_C(x) = W_{\Sim_2(a+1)}(x) \cdot W_{\RM_2(a+2)}(x)$.

Case 2: $P\notin\bar{\mathcal{C}}$, so  $P\in (V_1 + W_2) \setminus ((W_1 + W_2) \cup V_1)$.
A moment's reflection shows that all these choices for $P$ lead to equivalent point sets $\bar{\mathcal{C}}$.
As $P$ is not collinear with two different points of $\bar{\mathcal{C}}$, the projection $\mathcal{C}$ with respect to $P$ is a proper set and therefore, $C$ is projective.
So $\mathcal{C}$ is the disjoint union of the two affine subspaces $\pi_P(A_1)$ and $\pi_P(A_2)$ of dimension $a+1$.

The dimension formula leads to $\dim(\pi_P(V_1) \cap \pi_P(V_2)) = 1$.
There are unique points $Q_1\in V_1\setminus W_1$ and $Q_2\in W_2$ such that $P$ is the third point on the line $L = Q_1 + Q_2$
The affine space $\pi_P(A_2)$ has the hyperplane at infinity $(W_2 + P) / P$, which contains the single point $(Q_1 + P)/P = (Q_2 + P) / P = L/P$ of the affine space $\pi_P(A_1)$.
So the point $\pi_P(V_1) \cap \pi_P(V_2) = L/P$ is contained in $\pi_P(A_1)$ and in the hyperplane at infinity of $\pi_P(A_2)$.
This leads to the generator matrix stated in the theorem.

By construction, the code corresponding to the point set $\mathcal{C}$ is a projective $[2^{a+1}, 2a+3]_2$-code with (at most) the weights $\Delta$, $2\Delta$ and $3\Delta$.
Equations~\eqref{eq_A1}, \eqref{eq_A2} and \eqref{eq_A3} evaluate to the stated weight enumerator of $C$.
\end{proof}

Looking at the feasible cases in Table~\ref{table:q2}, we notice that all of them satisfy $w_2=n/2$, which corresponds to $\theta_2=0$,
$\theta_3=-\theta_1$ for the eigenvalues of $\SWRGparam$-SWRGs, see Equation~(\ref{eq_diophantine_vandam}). While we conjecture
that all integral solutions of Equation~(\ref{eq_diophantine_vandam}) satisfy this extra constraint for all $\SWRGparam\ge 5$, see
Section~\ref{sec_sum_monomials}, the condition $\theta_1+\theta_2+\theta_3=0$, i.e., $w_1+w_2+w_3=3n(q-1)/q$, is sufficient for
$\SWRGparam=3$.
So, it is an interesting open question, if $3$-SWRGs obtained from the coset graph of the dual code of a projective three-weight code also have to satisfy this extra condition.
To stimulate research into this direction we propose:

\begin{conjecture}
  \label{conj_w2_n}
  Let $C$ be a projective $[n,k]_2$ three-weight code with non-zero weights $w_1 < w_2 < w_3$ satisfying $w_1 + w_2 + w_3 = \tfrac{3n}{2}$.
  Then $w_2=\frac{n}{2}$.
\end{conjecture}

We remark that the MacWilliams identities, using the non-negativity and integrality constraints, are not sufficient to prove
Conjecture~\ref{conj_w2_n}.
As an example, the values $(n,w_1,w_2,w_3)\in\{(58,24, 31, 32), (68, 30, 32, 40)\}$ go in line with these conditions for $q=2$ but
are excluded with more sophisticated methods, see the details stated above. Given the results obtained so far we can state that
Conjecture~\ref{conj_w2_n} is true for all $n<72$. The next case, where all non-negativity and integrality constraints for
the $B_i$ are satisfied, is given by $(n,w_1,w_2,w_3)=(100,46, 48, 56)$. Here we have $k=7$, $A_{w_1}=32$, $A_{w_2}=145$, $A_{w_3}=78$, and
$B_3=580$. However, we can apply Lemma~\ref{lem:div_one_more} to conclude the non-existence of a binary linear code with
these parameters. More precisely, Lemma~\ref{lem:div_one_more}.\ref{div_one_more_case3}, applied with $a=1$ and $T=224$, yields
a contradiction since $T-2^k+2^{k-a}=96$ is not a power of two. In Table~\ref{tbl:counterex} we list all parameters
$\left(n,w_1,w_2,w_3,y=2^{k-2},A_1,A_2,A_3,B_3\right)$ up to $n=256$, where all $B_i$ are integral and non-negative and also
Lemma~\ref{lem:div_one_more} does not yield a contradiction, 
i.e., the parameters of potential counterexamples to Conjecture~\ref{conj_w2_n}:
\begin{table}
\caption{Parameters of potential counterexamples to Conjecture~\ref{conj_w2_n}}
\label{tbl:counterex}
\[
\begin{array}{rrrrrrrrr}
  n & w_1 & w_2 & w_3 & y=2^{k-2} & A_1 & A_2 & A_3 & B_3 \\
  \hline
  112 &  50 &  54 &  64 &  128 &  48 &  336 &  127 &  322 \\
  116 &  54 &  56 &  64 &  128 &  256 &  56 &  199 &  440 \\
  120 &  54 &  62 &  64 &  64 &  72 &  120 &  63 &  1180 \\
  124 &  56 &  64 &  66 &  64 &  72 &  119 &  64 &  1296 \\
  140 &  64 &  72 &  74 &  64 &  71 &  120 &  64 &  1840 \\
  202 &  96 &  103 &  104 &  64 &  67 &  128 &  60 &  5396 \\
  212 &  96 &  110 &  112 &  256 &  297 &  640 &  86 &  1860 \\
  212 &  96 &  110 &  112 &  512 &  649 &  896 &  502 &  1090 \\
  240 &  110 &  122 &  128 &  256 &  288 &  480 &  255 &  2450 \\
\end{array}
\]
\end{table}

\subsection{Feasible parameters for projective ternary three-weight codes with $\mathbf{w_1+w_2+w_3=2n}$}
\label{subsec_feasible_q_3}

In Table~\ref{table:q3} we list the admissible parameters for projective ternary three-weight codes with $w_1+w_2+w_3=2n$.
For each length $3\leq n\le 39$ we list the admissible dimensions $k$ and weight triples $(w_1,w_2,w_3)$, and
the weight distribution in the form $(A_{w_1},A_{w_2},A_{w_3})$.
The last column contains known results about the existence of codes with these parameters.
For some cases we can also state the number of isomorphism types of those codes.
If not mentioned otherwise, the classification
results are obtained with the software package \texttt{LinCode} \cite{bouyukliev2021computer}.
For the parameters marked with $\ge 1$ we constructed at least one code by prescribing an automorphism group, see~\cite{bkw2005}.

We also list those non-existence results where more sophisticated methods are necessary.
We mark the non-existence results with the keyword {\lq\lq}\textbf{None}{\rq\rq} in the comment column of Table~\ref{table:q3} and give a reference to the used method.

For $n=27$, $k=6$, and weight triple $(9, 18, 27)$ we have used exhaustive enumeration
using \texttt{LinCode} to exclude the existence of the corresponding code.
It would be nice to also have a theoretical argument.
For $36 \leq n\leq 39$ four cases remain undecided, which we mark with the keyword {\lq\lq}\textbf{Open}{\rq\rq}.
For each feasible case we give a suitable generator matrix in Appendix~\ref{appendix_generator_matrices}.

  \begin{center}
    \begin{longtable}[c]{lllll}
    \caption{Admissible and realizable parameters of ternary projective three-weight codes}\label{table:q3}\\
    $n$ & $k$ & $(w_1,w_2,w_3)$ & $(A_{w_1},A_{w_2},A_{w_3})$ & isomorphism types \\ \hline
    3 & $3$ & $(1,2, 3)$ & $(6,12,8)$ & 1 \\
    6 & $3$ & $(3, 4, 5)$ & $(8,6,12)$ & 1 \\
    9 & $3$ & $(5, 6, 7)$ & $(6,8,12)$ & 1 \\
    9 & $4$ & $(3, 6, 9)$ & $(6,66,8)$ & 1 \\
    18 & $4$ & $(9, 12, 15)$ & $(8,60,12)$ & 4 \\
    18 & $5$ & $(9, 12, 15)$ & $(44,150,48)$ & 213 \\
    18 & $6$ & $(9, 12, 15)$ & $(152,420,156)$ & 52 \\
    27 & $4$ & $(15, 18, 21)$ & $(6,62,12)$ & 2 \\
    27 & $5$ & $(15, 18, 21)$ & $(60,116,66)$ & $\ge$ 2\,695\,546 \\
    27 & $6$ & $(15, 18, 21)$ & $(222,278,228)$ & 6 \\
    27 & $5$ & $(9, 18, 27)$ & $(6,228,8)$ & 1  \\
    27 & $6$ & $(9, 18, 27)$ & $(24,678,26)$ & \textbf{None} exhaustive enumeration\\
    36 & $5$ & $(21, 24, 27)$ & $(72,90,80)$ & $\ge$ 1\\
    36 & $6$ & $(21, 24, 27)$ & $(288,144,296)$ & $\ge$ 1\\
    36 & $7$ & $(21, 24, 27)$ & $(936,306,944)$ & \textbf{Open}\\
    39 & $5$ & $(21, 27, 30)$ & $(42,188,12)$ & \textbf{Open}\\
    39 & $6$ & $(21, 27, 30)$ & $(156,494,78)$ & \textbf{Open}\\
    39 & $7$ & $(21, 27, 30)$ & $(498,1412,276)$ & \textbf{Open}\\
    \end{longtable}
  \end{center}

  Similar to Conjecture~\ref{conj_w2_n}, the numerical data suggests the conjecture $w_2 = \frac{2}{3} n$.
  Based on our computational data, we dare to state the following $q$-ary version of Conjecture~\ref{conj_w2_n}.

\begin{conjecture}
  \label{conj_gen}
  Let $C$ be a projective $[n,k]_q$ three-weight code with non-zero weights $w_1<w_2<w_3$ satisfying $w_1+w_2+w_3=3 (1 - \frac{1}{q}) n$.
  Then $w_2=(1 - \frac{1}{q}) n$.
  Moreover, $w_1 = w_2 - t$ and $w_3 = w_2 + t$, where $t$ is a power of the characteristic $p$ of $\mathbb{F}_q$.
\end{conjecture}

For $q=2$, Conjecture~\ref{conj_gen} follows from Conjecture~\ref{conj_w2_n} by Lemma~\ref{lem:t_divisibility}.
We further remark that the precondition $w_1 + w_2 + w_3 = 3 (1 - \frac{1}{q}) n$ cannot be dropped, as seen by the binary $[7,4]_2$ Hamming code, which is a three-weight code with weight distribution $(0^1 3^3 4^3 7^1)$.

\section{Plane curves given by the sum of all monomials of given degree}
\label{sec_sum_monomials}

In this section, we present some results on rational (or integral) solutions
of the equation
\begin{equation} \label{E:curve}
  \sum_{h+i+j = \SWRGparam-2} \theta_1^h \theta_2^i \theta_3^j = 0 \,,
\end{equation}
which for pairwise distinct $\theta_1, \theta_2, \theta_3$ is equivalent
to~\eqref{eq_diophantine_vandam} by Lemma~\ref{lem:reltheta}. We restrict
to the case that $\SWRGparam$ is odd. When $\SWRGparam$ is even, then
there are no nontrivial real solutions, so \emph{a fortiori} no rational solutions.

We denote by $C_{\SWRGparam-2}$ the plane projective curve defined by~\eqref{E:curve},
and we will rename the variables $\theta_1, \theta_2, \theta_3$ in this section
as $x, y, z$. As already mentioned, $C_1$ is the line $x + y + z = 0$, and there
are many rational points on this curve. In general, it is not hard to see
that $C_d$ is smooth over~$\Q$, so the curve is in particular geometrically
irreducible and has genus $g(C_d) = (d-1)(d-2)/2$.

For $d = 3$ (corresponding to $\SWRGparam = 5$), $C_3$ is a curve of genus~$1$
with some rational points, so it is an elliptic curve. A standard procedure
(implemented, for example, in Magma~\cite{Magma}) produces an isomorphic curve
in Weierstrass form. It turns out that $C_3$ is isomorphic to the curve with
label~$50a1$ in the Cremona database
(\href{http://www.lmfdb.org/EllipticCurve/Q/50/a/3}{$50.a3$} in the LMFDB \cite{lmfdb:50.a3}).
In Cremona's tables or under the link above, one can check that this curve has
exactly three rational points. This proves the following.

\begin{lemma} \label{lem:d=3}
  \[ C_3(\Q) = \{(1:-1:0), (-1:0:1), (0:1:-1)\} \,. \]
\end{lemma}

The curve~$C_5$ is a plane quintic of genus~$6$. Note that there is an action
of the symmetric group~$S_3$ on three letters on every curve~$C_d$ by permuting
the coordinates. We can restrict this action to an action of the subgroup~$A_3$
generated by a cyclic permutation. The quotient~$C'_5$ of~$C_5$ by this action of~$A_3$
is a curve of genus~$2$. We can compute a singular plane model of~$C'_5$
by taking the image of~$C_5$ under the map
\[ \BP^2 \to \BP^2\,,\quad
   (x:y:z) \mapsto (xyz : (xy+yz+zx)(x+y+z) : (x-y)(y-z)(z-y)) \,.
\]
A procedure implemented in Magma~\cite{Magma} then produces the hyperelliptic equation
\[ H_5 \colon y^2 = -3 x^6 + 8 x^5 - 28 x^4 - 30 x^3 + 40 x^2 + 16 x - 15 \]
for~$C'_5$.
A 2-descent as described in~\cite{Stoll2001} (and implemented in Magma)
shows that the Mordell-Weil rank of the Jacobian~$J$ of~$H_5$ is at
most~$1$. Since one finds a point on~$J$ of infinite order (with Mumford
representation $(x^2-x+2, 7x+7)$), the rank is indeed~$1$.
Using the Magma implementation of Chabauty's method combined with the
Mordell-Weil sieve (see~\cite{BruinStoll2010}), one quickly finds that
the only rational point on this hyperelliptic curve is~$(-1,0)$. This point must
be the image of the three obvious rational points on~$C_5$. Since any
other rational point would have to map to a different point on~$H_5$,
this proves the following.

\begin{lemma} \label{lem:d=5}
  \[ C_5(\Q) = \{(1:-1:0), (-1:0:1), (0:1:-1)\} \,. \]
\end{lemma}

The combination of Lemma~\ref{lem:d=3} and Lemma~\ref{lem:d=5} with Lemma~\ref{lem:vandam_omidi} leads to the following theorem.

\begin{theorem}
Let $\Gamma$ be a $k$-regular graph with four distinct eigenvalues $k>\theta_1>\theta_2>\theta_3$ and let $\SWRGparam\in\{5,7\}$.
Then $\Gamma$ is an $\SWRGparam$-SWRG if and only if $\theta_2=0$ and $\theta_3=-\theta_1$.
\end{theorem}

Considering larger odd~$d$, we can say the following. The quotient~$C''_7$
of~$C_7$ by the full $S_3$-action is an elliptic curve, which is isomorphic
to the curve with label~$10368w1$ in the Cremona database
(\href{http://www.lmfdb.org/EllipticCurve/Q/10368/j/1}{$10368.j1$}
in the LMFDB \cite{LMFDB:10368.j1}). Unfortunately, this curve has rank~$2$ and therefore
has infinitely many rational points. So we cannot use this approach
to determine the set of rational points on~$C_7$.

The quotient $C''_9$ of~$C_9$ by the $S_3$-action is a smooth plane quartic curve, isomorphic
to the curve with equation
\begin{align*}
  x^4 &+ 2 x^3 y + x^2 y^2 - x y^3 - y^4 + 2 x^3 z - 4 x^2 y z - 3 x y^2 z \\
      &+ 2 y^3 z + 4 x^2 z^2 - 3 x y z^2 + 3 y^2 z^2 + 3 x z^3 - 4 y z^3 + z^4 = 0 \,.
\end{align*}
A point search finds the two rational points $(-5:1:4)$ and $(-1:1:0)$.
The first is the image of the three obvious rational points on~$C_9$,
whereas the second point does not lift to a rational point on~$C_9$.
Let $J$ be the Jacobian of the curve. Then $\#J(\F_3) = 3^3$
and $\#J(\F_7) = 11 \cdot 31$, so $J(\Q)$ has trivial torsion subgroup.
Therefore, the point in~$J(\Q)$ given by the difference of the two
rational points has infinite order. It might be possible to use the
methods of~\cite{BPS2016} to determine the rank of~$J(\Q)$.
If the rank turns out to be $\le 2$, then an application of Chabauty's
method might show that the two known rational points are the only ones.

In any case, searching for rational points does not exhibit any other points
than the obvious ones when $d \ge 3$ is odd. This leads to the following conjecture,
which generalizes the results of Lemma~\ref{lem:d=3} and~\ref{lem:d=5}.

\begin{conjecture}
  If $d \ge 3$ is odd, then
  \[ C_d(\Q) = \{(1:-1:0), (-1:0:1), (0:1:-1)\} \,. \]
  Equivalently, all solutions $(\theta_1, \theta_2, \theta_3)$ in integers
  of~\eqref{eq_diophantine_vandam} with $\SWRGparam \ge 5$ odd and
  $\theta_1 > \theta_2 > \theta_3$ satisfy $\theta_2 = 0$ and $\theta_3 = -\theta_1$.
\end{conjecture}


\section{Divisibility for binary linear codes with few weights}
\label{sec_divisibility}

In this section we want to study the divisibility properties of the weights and the length of the binary linear codes with few weights. A first but
very powerful tool are the MacWilliams identities.
In Equations~(\ref{eq_ppm1proj})-(\ref{eq_ppm4proj}), the code has been assumed to be projective, i.e. $B_2 = 0$.
For the more general situation, we prepare
\begin{eqnarray}
  \sum_{i>0} A_i &=& q^k-1,\label{eq_ppm1}\\
  \sum_{i\ge 0} iA_i &=& q^{k-1}n,\label{eq_ppm2}\\
  \sum_{i\ge 0} i^2A_i &=& q^{k-1}(B_2+n(n+1)/2),\label{eq_ppm3}\\
  \sum_{i\ge 0} i^3A_i &=& q^{k-2}(3(B_2n-B_3)+n^2(n+3)/2)\label{eq_ppm4},
\end{eqnarray}
for an [$n,k]_q$ code $C$ with $B_1=0$.

\begin{lemma}[Folklore]\label{lemma_even_weight_subcode}
  Let $C$ be an $[n,k]_2$ code and $C_2$ the subset of all codewords of even weight.
  Then $C_2$ is a linear subcode of $C$ of dimension $k$ or $k-1$.
\end{lemma}

\begin{proof}
Consider the $\F_2$-linear map $f : C \to \F_2$, $c \mapsto \sum_{i=0}^n c_i$.
Then $C_2 = \ker f$ is a linear subspace of $C$.
By the rank-nullity theorem, the codimension of $C_2$ in $C$ equals $\dim \ker f \in \{0,1\}$.
\end{proof}

We call $C_2$ the \emph{even-weight subcode} of $C$.

A direct consequence of Lemma~\ref{lemma_even_weight_subcode} is the following.

\begin{lemma}
  \label{lemma_swt_odd}
  Let $C$ be an $[n,k]_2$ code of dimension $k \geq 2$.
  Then $C$ has a non-zero even weight.
\end{lemma}

\begin{lemma}
	\label{lem:3wt_dim}
	Let $C$ be a linear binary $[n,k]_2$ three-weight code.
	Then $k \geq 2$.
	If $C$ is projective, then $k \geq 3$.
\end{lemma}

\begin{proof}
	if $k \leq 1$, then $C$ consists of at most a single non-zero codeword, so $C$ cannot have three different weights.

	Assume that $C$ is projective of dimension $k=2$ and let $G$ be a generator matrix of $C$.
	Then $G$ neither has a zero column, nor a repeated column.
	Therefore, each of the $2^k - 1 = 3$ possible column vectors in $\F_2^2\setminus\{\mathbf{0}\}$ appears at most once as a column of $G$, implying that $n \leq 3$.
	As $C$ has three different non-zero weights, $n \geq 3$, so together we get $n = 3$ and each of the $3$ non-zero vectors appears exactly once as a column of $G$.
	Therefore, $C$ is isomorphic to the simplex code $\Sim_2(2)$, which is a constant weight code.
	This is a contradiction.
\end{proof}

\begin{remark}
	There are indeed (many) non-projective binary three-weight codes of dimension $2$.
	An example for the smallest possible length $n = 3$ is given by the generator matrix $(\begin{smallmatrix}1 & 0 & 0 \\ 0 & 1 & 1\end{smallmatrix})$, which spans a code with the weight distribution $(0^1 1^1 2^1 3^1)$.
\end{remark}

\begin{lemma}
  \label{lemma_3wt_neven_1oddweight}
  Let $C$ be a projective full-length $[n,k]_2$ three-weight code with non-zero weights $w_1$, $w_2$ and $w_3$, such that $n$ is even and exactly one weight is odd.
  W.l.o.g.~let $w_2$ be the odd weight.

 Then $w_2=\frac{n}{2}$ and the even-weight subcode $C_2$ of $C$ has effective length $n$ and is a $2$-fold replication of a projective
 $\left[\tfrac{n}{2},k-1\right]_2$ two-weight code
  with non-zero weights $\frac{w_1}{2}$ and $\frac{w_3}{2}$.
\end{lemma}

\begin{proof}
  Let $A_{w_i}$ be the number of codewords of weight $w_i$ in $C$.
  Furthermore, let $(B_i)$ be the dual weight distribution of $C$ and $(B'_i)$ the dual weight distribution of $C_2$.
  Since $C$ is projective, we have $B_1 = B_2 = 0$.
  We set $y=2^{k-2} = \frac{1}{4}\#C$.
  By Lemma~\ref{lem:3wt_dim}, $y\in \Z$.
  Since $w_2$ is the only odd weight, Lemma~\ref{lemma_even_weight_subcode} gives $A_{w_2} = 2y$.
  Now Equation~\eqref{eq_ppm2proj} applied to $C$ yields
  \[
	  w_1 A_{w_1} + w_3 A_{w_3} = 2y(n-w_2)\text{.}
  \]
  From Lemma~\ref{lemma_even_weight_subcode} we conclude that $C_2$ is a two-weight code of dimension $k-1$ and effective length $n' \leq n$ with non-zero weights $w_1$ and $w_3$.
  Since $C$ is projective, we have $n'\in\{n-1,n\}$.
  Noting that $A_{w_1}$ and $A_{w_3}$ are also the numbers of codewords of weights $w_1$ and $w_3$ in $C_2$, the application of equation~\eqref{eq_ppm2} to the full-length code arising from $C_2$ after (possibly) removing the zero position yields
  \begin{eqnarray}
    w_1A_{w_1}+w_3A_{w_3} &=& n'y\text{.}\label{eq_subcode}
  \end{eqnarray}
  Hence $n'y = 2y(n-w_2)$ and thus $n' = 2(n - w_2)$ is even.
  By the assumtion that $n$ is even, $n' = n-1$ is not possible.
  Therefore $n'=n$ and $w_2=\frac{n}{2}$.
  So $C_2$ is full-length and hence $B'_1 = 0$.
  Now the difference of Equations~\eqref{eq_ppm3} for $C$ and $C_2$ with $w_2 = \frac{n}{2}$, $A_{w_2} = 2y$ and $B_2 = 0$ gives
  \[
  	\frac{n^2}{4} \cdot 2y = 2y \cdot \frac{n(n+1)}{2} - y \left(B'_2 + \frac{n(n+1)}{2}\right)\text{,}
  \]
  which simplifies to $\frac{n^2}{2} = \frac{n(n+1)}{2} - B'_2$ and further to $B'_2 = \frac{n}{2}$.
  As $C$ is projective, the position multiplicities of the codimension $1$ subcode $C_2$ are at most $2$.
  Denoting the number of position pairs of multiplicity $2$ by $m$, Lemma~\ref{lem:col_mult_B2} yields $m = B_2' = \frac{n}{2}$.
  Therefore, all positions of $C_2$ appear with multiplicity $2$ and thus, $C_2$ is the two-fold repetition of a projective two-weight code with non-zero weights $\frac{w_1}{2}$ and $\frac{w_3}{2}$.
\end{proof}

\begin{remark}
	\label{rem:w1_w3}
	As seen in the above proof, in the situation of Lemma~\ref{lemma_3wt_neven_1oddweight} we have $A_{w_2} = 2y$.
	Moreover, we can use Equations~\eqref{eq_ppm1} and~\eqref{eq_ppm2} to compute the frequencies
	\[
		A_{w_1} = \frac{(2y-1) w_3 - yn}{w_3 - w_1}
		\quad\text{and}\quad
		A_{w_3} = \frac{yn - (2y - 1)w_1}{w_3 - w_1}
	\]
	depending on the weights $w_1$ and $w_3$.
\end{remark}

From now on, we add the extra constraint $w_1+w_2+w_3=\tfrac{3n}{2}$.

\begin{corollary}
  \label{cor_n_divisible_by_4}
  Let $C$ be a projective $[n,k]_2$ three-weight code with non-zero weights $w_1$, $w_2$ and $w_3$ satisfying $w_1+w_2+w_3=\tfrac{3n}{2}$. Then $n$ is a multiple of $4$.
\end{corollary}

\begin{proof}
  Since $\tfrac{3n}{2} = w_1+w_2+w_3$ is an integer, $n$ has to be even, so that we assume $n\equiv 2\pmod 4$.
  Then $\tfrac{3n}{2}=w_1+w_2+w_3$ is odd.
  By Lemma~\ref{lemma_swt_odd}, $C$ has an even weight, so exactly one weight of $C$ is odd.
  Without restriction, let $w_2$ be the odd weight.
  Lemma~\ref{lemma_3wt_neven_1oddweight} yields $w_2 = \frac{n}{2}$.
  From $w_1+w_2+w_3=\frac{3}{2} n$ we may further assume $w_1 < w_2 < w_3$, so $w_1=\tfrac{n}{2}-t$ and $w_3=\tfrac{n}{2}+t$ for some positive integer $t$.
  Since $w_1$ and $w_3$ are even, $t$ has to be odd.
  Moreover, Lemma~\ref{lemma_3wt_neven_1oddweight} says that $\frac{w_1}{2}$ and $\frac{w_3}{2}$ are the weights of a projective binary two-weight code. By
  Lemma~\ref{lemma_delsarte} the weight difference $\tfrac{w_3}{2}-\tfrac{w_1}{2}=t$ has to be a power of $2$. Since $t$ is odd, we conclude that $t=1$, so
  \[
  	w_1 = \frac{n}{2}-1\text{,}\quad
	w_2 = \frac{n}{2}\quad\text{and}\quad
	w_3 = \frac{n}{2}+1\text{.}
  \]
  Writing $y = 2^{k-2}$, the frequencies from Remark~\ref{rem:w1_w3} evaluate to
  \begin{equation}
  	\label{eq:w1_w2_w3_special}
	  A_{w_1} = y -\frac{n}{4} - \frac{1}{2}\text{,}\quad
	  A_{w_2} = 2y \quad\text{and}\quad
	  A_{w_3} = y + \frac{n}{4} - \frac{1}{2}\text{.}
  \end{equation}
  Plugging these expressions into Equation~\eqref{eq_ppm3proj} leads to
  \[
  	n^2y + 2y + \frac{n^2}{4} - 1 = n(n+1)y
  \]
  and further to the quadratic equation
  \[
	  n^2 - 4ny + (8y - 4) = 0
  \]
  with the two solutions $n \in \{2, 4y - 2\}$.
  Since the length of a three-weight code is at least $3$, necessarily $n=4y-2$.
  Now Equation~\eqref{eq:w1_w2_w3_special} yields $A_{w_1}=0$ -- a contradiction.
\end{proof}

Using the abbreviation $y=2^{k-2}$, we prepare equations \eqref{eq_A1}--\eqref{eq_B3} in the special case $w_1=\tfrac{n}{2}-t$, $w_2 = \frac{n}{2}$ and $w_3=\tfrac{n}{2}+t$.
  \begin{eqnarray}
    A_{w_1} &=& \frac{n(4y - n - 2t)}{8t^2} \label{eq:A1_pmt}\\
    A_{w_2} &=& 4y - 1 - \frac{n(4y-n)}{4t^2} \label{eq:A2_pmt}\\
    A_{w_3} &=& \frac{n(4y - n +2t)}{8t^2}\label{eq:A3_pmt}\\
    3B_3&=&  \frac{n(n-2t)(n+2t)}{8y}\label{eq:B3_pmt}
  \end{eqnarray}

\begin{lemma}
	\label{lem:3wt_k3}
	Let $C$ be a projective $[n,3]_2$ three-weight code with non-zero weights $w_1 < w_2 < w_3$ satisfying $w_1+w_2+w_3=\tfrac{3n}{2}$.
	Then $C$ has length $n=4$, weight distribution $(0^1 1^1 2^3 3^3)$ and is isomorphic to the code spanned by the generator matrix
  \[
  	\left(\begin{matrix}
      1&0&0&0\\
      0&1&0&1\\
      0&0&1&1
    \end{matrix}\right)\text{.}
\]
\end{lemma}

\begin{proof}
	By Lemma~\ref{lem:3wt_length_restriction}, $n \leq 5$, and by Corollary~\ref{cor_n_divisible_by_4}, $4 \mid n$.
	Therefore $n = 4$.
	The code $C$ is isomorphic to a systematic code, which has a generator matrix of the form $(I_3 \mid v)$, where $I_3$ denotes the $3\times 3$ unit matrix and $v$ is a vector in $\F_2^3$.
	As $C$ is projective, $w(y) \geq 2$.
	Furthermore, $w(y) = 3$ is not possible, as $C$ would have only the two weights $2$ and $4$.
	So $w(y) = 2$.
	We note that the three possibilities for $y$ lead to equivalent codes, and that the resulting code has the stated weight distribution.
\end{proof}

We remark that geometrically, the above $[4,3]_2$ code corresponds to the complement of a triangle in the projective plane $\PG(2,2)$.

\begin{theorem}
  \label{thm:3wt_thm5_no_weight_odd}
  Let $C$ be a projective $[n,k]_2$ three-weight code with non-zero weights $w_1 < w_2 < w_3$ satisfying $w_1+w_2+w_3=\tfrac{3n}{2}$.
  Then $n$ is a multiple of $4$, and one of the following cases occurs.
  \begin{enumerate}[(i)]
  	\item $k\geq 4$, $n\geq 8$ and all weights of $C$ are even.
	\item The code $C$ has the parameters $[4,3]_2$ and is isomorphic to the code in Lemma~\ref{lem:3wt_k3}.
  \end{enumerate}
\end{theorem}

\begin{proof}
  By Corollary~\ref{cor_n_divisible_by_4}, $n$ is a multiple of $4$.

  In the case that $n$ has only even weights, the largest weight is at least $6$, so $n \geq 8$.
  Moreover, $k \geq 4$ by Lemma~\ref{lem:3wt_k3}.

  Now we assume that $C$ has an odd weight.
  As $\frac{n}{2}$ is even, $C$ has at least two odd weights by Lemma~\ref{lemma_3wt_neven_1oddweight}.
  Since $w_1 + w_2 + w_3 = n$ is even, we get that $C$ has exactly two odd weights, say $w_1$ and $w_3$.
  Let $C_2$ be the even-weight subcode of $C$.
  The code $C$ is projective and by Lemma~\ref{lemma_even_weight_subcode}, the codimension of $C_2$ in $C$ is $1$.
  Therefore, the maximum position multiplicity of $C_2$ is at most $2$, and the effective length $n'$ of $C_2$ is either $n-1$ or $n$.
  Since $w_2$ is the only even weight of $C$, the subcode $C_2$ is a code of constant weight $w_2$ and frequency $A_{w_2}=\#C_2 - 1 = 2^{k-1}-1$.
  From Lemma~\ref{lemma_1wt} we conclude that $w_2 = u\cdot 2^{k-2}$ and $n'=u\cdot (2^{k-1}-1)$ with an integer $u\in\{1,2\}$, where $u \leq 2$ follows from the maximum position multiplicity.

  Let us first investigate the case $u=2$.
  Here, $n'=2^k-2$ and $n \in \{2^k-2, 2^k-1\}$.
  Let $G$ be a generator matrix of $C$.
  Since $C$ is projective, $G$ neither has a zero column, nor a repeated column.
  So each of the $2^k-1$ non-zero vectors in $\F_2^k$ occurs exactly once as a column in $G$, possibly with the excection that a single vector might not occur at all.
  In the case $n = 2^k - 1$, all vectors occur as a column in $G$, so $C = \Sim_2(k)$, which is a code of constant weight $2^{k-1}$.
  In the case $n = 2^k - 2$, $C$ is the simplex code $\Sim_2(k)$ punctured in a single position, so $C$ has the two weights $2^{k-1}$ and $2^{k-1}-1$.%
  \footnote{In fact, in this case $C$ is a MacDonald code.}
  This contradicts the assumption that $C$ is a three-weight code.

  It remains to consider $u=1$.
  Here, $w_2 = 2^{k-2}$ and $n' = 2^{k-1} - 1$, which is odd.
  Since $n$ is even, necessarily $n' = n-1$, so $n = 2^{k-1}$ and $w_2 = \frac{n}{2}$.
  Combined with $w_1 + w_2 + w_3 = \frac{3n}{2}$ we can write $w_1=\tfrac{n}{2}-t$ and $w_3=\tfrac{n}{2}+t$ for some positive integer~$t$.
  Together with the abbreviation $y = 2^{k-2}$, equations~(\ref{eq:A1_pmt})--(\ref{eq:A3_pmt}) can be simplified to
  \begin{eqnarray*}
    A_{w_1} &=& \frac{y(y-t)}{2t^2}\\
    A_{w_2} &=& 4y - 1 - \frac{y^2}{t^2} \\
    A_{w_3} &=& \frac{y(y+t)}{2t^2}.
  \end{eqnarray*}
  Now we use $A_{w_2}=2^{k-1}-1=2y-1$ to conclude $y=2t^2$ (or $y=0$, which is impossible).
  As $y$ is a power of $2$, so is $t$.
  From $w_1 = t(2t - 1)$ odd we get that $t$ is is odd.
  Together, this forces $t = 1$, which gives $y = 2$ and therefore $k = 3$.
  Therefore, $C$ is isomorphic to the code in Lemma~\ref{lem:3wt_k3}.
\end{proof}

\begin{lemma}
	\label{lem:t_divisibility}
	Let $C$ be a projective $[n,k]_2$ three-weight code with non-zero weights $w_1 < w_2 < w_3$ satisfying $w_1+w_2+w_3=\tfrac{3n}{2}$ and $w_2 = \frac{n}{2}$.
	Then $w_1 = w_2 - t$ and $w_3 = w_2 + t$ where $t$ is a power of $2$.
	Moreover, $2t \mid n$, and $t$ is the largest integer $\Delta$ such that $C$ is $\Delta$-divisible.
\end{lemma}

\begin{proof}
	Equations~\eqref{eq:A1_pmt} and~\eqref{eq:A3_pmt} give $A_{w_3} - A_{w_1} =\tfrac{n}{2t}$, so $2t \mid n$.
	Therefore $t \mid w_2 = \frac{n}{2}$, implying that $t$ is the greatest common divisor of $w_1 = w_2 - t$, $w_2$ and $w_3 = w_2 + t$, so $\Delta = t$.
	Since $C$ is projective, $C$ cannot be the proper repetition of some code.
	Now as a consequence of \cite[Theorem 1]{ward1981divisible}, the number $\Delta = t$ must be a power of $2$.
\end{proof}

\begin{lemma}
  \label{lem:divisibility_2}
  Let $C$ be a projective $[n,k]_2$ three-weight code with non-zero weights $w_1 < w_2 < w_3$ satisfying $w_1+w_2+w_3=\tfrac{3n}{2}$ and $w_2=\tfrac{n}{2}$.
  Let $a \geq 0$ be the largest integer such that $C$ is $2^a$-divisible.
  Then $k \leq 8a + 9$.
\end{lemma}

\begin{proof}
As before, we will use the abbreviation $y=2^{k-2}$ and write $w_1=\tfrac{n}{2}-t$ and $w_3=\tfrac{n}{2}+t$ with an integer $t \geq 1$.
By Lemma~\ref{lem:3wt_dim}, $y\in\Z$.
By Lemma~\ref{lem:t_divisibility}, $t = 2^a$ is the largest integer $\Delta$ such that $C$ is $\Delta$-divisible.
There is an odd integer $z$ and a non-negative integer $x$ with $n = 2^x \cdot z$.
Since $C$ is projective, $n \leq 2^k - 1$.
Together with $2t \mid n$, we get $a+1\leq x\leq k-1$.

Plugging $t = 2^a$, $n = 2^x z$ and $y = 2^{k-2}$ into Equations~\eqref{eq:A1_pmt}--\eqref{eq:B3_pmt} we get
\begin{eqnarray}
	A_{w_1} &=& \frac{z\cdot (2^{k-a-1}-2^{x-a-1} z - 1)}{2^{a+2-x}}\text{,} \label{eqs1}\\
	A_{w_2} &=& 2^{2(x-a-1)}z^2 + 2^{k} - 2^{x+k-2a-2}z - 1\text{,} \label{eqs2}\\
	A_{w_3} &=& \frac{z\cdot (2^{k-a-1}-2^{x-a-1}z + 1)}{2^{a+2-x}}\text{,} \label{eqs3}\\
	3B_3&=&  \frac{z\cdot \left(2^{x-a-1}z - 1\right)\cdot\left(2^{x-a-1}z + 1\right)}{2^{k-x-2a-1}}\text{.}\label{eqs4}
\end{eqnarray}

First case: $k \geq x + 2a + 2$.
Equivalently, $k-x-2a-1\ge 1$, so the denominator of the right hand side of Equation~\eqref{eqs4} is even.
By $B_3\in\mathbb{Z}$, the numerator is even, too.
Since $z$ is odd, this implies $a = x-1$.
Now $0 < w_1 = \frac{n}{2}-t = 2^{x-1} z - 2^a = 2^a (z - 1)$ yields $z > 1$.
Equation~\eqref{eqs4} with $a = x-1$ yields
\[
	3 B_3 = \frac{z\cdot \left(z - 1\right)\cdot\left(z + 1\right)}{2^{k-3x+1}}\text{.}
\]
From our precondition $k \geq x + 2a + 2 = 3x$ we have $k - 3x\geq 0$, so $2^{k-3x}$ is an integer.
Since $\gcd(z-1,z+1)=2$, we have that $2^{k-3x}$ either divides $z-1$ or $z+1$.
Therefore, $z=s\cdot 2^{k-3x}+\alpha$ for some integer $s$ and $\alpha\in\{-1,1\}$.
By $z > 1$, $s \geq 1$.
Now Equation~\eqref{eqs1} yields
\[
	0 < \frac{2^{a+2-x}}{z}\cdot A_{w_1} = 2^{k-x}-\left(s\cdot 2^{k-3x}+\alpha+1\right) \leq 2^{k-x}-s\cdot 2^{k-3x}\text{,}
\]
so that $s<2^{2x}$ and hence $s\le 2^{2x}-1$.

Using Equation~\eqref{eqs2} we get
\begin{align*}
	0
	& < s (A_{w_2} + 1) \\
	& = s (z^2 + 2^k) + s\cdot z \cdot 2^{k-x} \\
	& \leq (2^{2x} - 1)((s 2^{k-3x} + \alpha)^2 + 2^k) - s(s2^{k-3x} + \alpha) 2^{k-x} \\
	& = (2^{2x} - 1)(s^2 2^{2k-6x} + 2\alpha s 2^{k-3x} + 1 + 2^k) - s(s2^{k-3x} + \alpha) 2^{k-x} \\
	& = s^2 2^{2k-4x} + 2\alpha s 2^{k-x} + 2^{2x} + 2^{k+2x} - s^2 2^{2k-6x} - \alpha s 2^{k-3x+1} - 1 - 2^k \\
	& \qquad - s^2 2^{2k-4x} - \alpha s 2^{k-x} \\
	& = \alpha s (2^{k-x} - 2^{k-3x+1}) + (2^{2x} + 2^{k+2x}) - s^2 2^{2k-6x} - (1 + 2^k) \\
	& \leq 2s\cdot 2^{k+2x} + 2s\cdot 2^{k + 2x} - s 2^{2k-6x} \\ 
	& = s(2^{k+2x+2} - s^{2k - 6x}) \text{,}
\end{align*}
where in the second last step $s \geq 1$ has been used.
Therefore $k + 2x + 2 > 2k  - 6x$ and hence
\[
	k \leq 8x + 1 = 8(a + 1) + 1 = 8a + 9\text{.}
\]

Second case: $k\le  x+2a+1$.
Equation~\eqref{eq:A1_pmt} implies
\[
	0 < \frac{8t^2}{n} A_{w_1} = 4y - n - 2t < 4y - n\text{.}
\]
We have $2^x \mid n$ and from $x \leq k-1$ also $2^x \mid 2^k = 4y$.
Therefore $2^x \mid 4y - n > 0$ and thus $4y - n \geq 2^x$.

By Equation~\eqref{eq:A2_pmt},
\begin{align*}
	0
	& < 4t^2 A_{w_2} \\
	& = 4t^2(4y - 1) - n(4y-n) \\
	& \leq 4t^2\cdot 4y - n\cdot 2^x \\
	& = 2^{k + 2a + 2} - 2^{2x} z \\
	& \leq (2^{k + 2a + 2} - 2^{2x})
\end{align*}
and thus $k + 2a + 2 > 2x$ and $k \geq 2x - 2a - 1$.
Combined with $k\le  x+2a+1$, finally
\[
	k = 2k - k \leq 2(x + 2a + 1) - (2x - 2a - 1) = 6a + 3 < 8a + 9\text{.}
\]
\end{proof}

By Lemma~\ref{lem:divisibility_2} and Lemma~\ref{lem:3wt_length_restriction}, the following numbers $K(r)$ and $N(r)$ are well-defined.

\begin{definition}
	Let $a \geq 1$ be an integer.
	We define $K(a)$ (resp. $N(a)$) as the largest dimension (resp. length) of a projective $[n,k]_2$ three-weight code $C$ with non-zero weights $w_1 <  w_2 < w_3$ satisfying $w_1+w_2+w_3=\tfrac{3n}{2}$ and $w_2=\tfrac{n}{2}$ such that $C$ is not $2^a$-divisible.
\end{definition}

\begin{theorem}
	\label{thm:K_N_bounds}
	Let $a \geq 1$ be an integer.
	Then
	\[
		2a + 1 \leq K(a) \leq 8a + 1
		\qquad\text{and}\qquad 2^{2a + 1} - 2^{a+1} \leq N(a) \leq 2^{K(a)} - 3\text{.}
	\]
\end{theorem}

\begin{proof}
	$K(a) \leq 8(a-1) + 9 = 8a + 1$ by Lemma~\ref{lem:divisibility_2} and $N(a) \leq 2^{K(a)} - 3$ by Lemma~\ref{lem:3wt_length_restriction}.

	For the lower bounds, let $C = \Sim_2(a) \oplus \RM_2(a+1)$.
	It is a projective three-weight code of length $n = (2^a - 1) + 2^a = 2^{a+1} - 1$, dimension $k = a + (a+1) = 2a + 1$ and weights $w_1 = 2^{a-1}$, $w_2 = 2^a$ and $w_3 = 3\cdot 2^{a-1}$.%
	\footnote{We have already seen the code $C$ in Theorem~\ref{thm:4Delta} Case~\ref{thm:4Delta:minus1}, with $a-1$ in instead of $a$.}%
	Since $C$ does not have the weight $2^{k-1} = 2^{2a}$, the anticode $C^\complement$ is defined.
	It is a projective three-weight code of the same dimension $k$, length $n^\complement = (2^k - 1) - n = 2^{2a + 1} - 2^{a+1}$ and the three non-zero weights
	\begin{align*}
	w^\complement_1 & = 2^{k-1} - w_3 = 2^{2a} - 3\cdot 2^{a-1} = 2^{a-1}(2^{a+1} - 3)\text{,} \\
	w^\complement_2 & = 2^{k-1} - w_2 = 2^{2a} - 2^a = 2^{a-1}(2^{a+1} - 2)\quad\text{and} \\
	w^\complement_3 & = 2^{k-1} - w_1 = 2^{2a} - 2^{a-1} = 2^{a-1}(2^{a+1} - 1)\text{,}
	\end{align*}
	so $C^\complement$ is $2^{a-1}$-divisible, but not $2^a$-divisible.
	Furthermore
	\[
		w_1^\complement + w_2^\complement + w_3^\complement = 3\cdot 2^{2a} - 3\cdot 2^a = \frac{3}{2} n^\complement
		\qquad\text{and}\qquad
		w_2^\complement = \frac{n^\complement}{2}\text{.}
	\]
	Therefore, $K(a) \geq k = 2a + 1$ and $N(a) \geq n^\complement = 2^{2a + 1} - 2^{a+1}$.
\end{proof}

For small values of $a$, we can determine the exact values of $N(a)$ and $K(a)$.

\begin{theorem}\label{them:K_N_a_small}\item
	\begin{enumerate}[(a)]
		\item $K(1) = 3$ and $N(1) = 4$.
		\item $K(2) = 6$ and $N(2) = 56$.
		\item $K(3) = 11$ and $N(3) = 2024$.
	\end{enumerate}
\end{theorem}

\begin{proof}
	The case $a = 1$:
	The values $K(1) = 3$ and $N(1) = 4$ are a direct consequence of Theorem~\ref{thm:3wt_thm5_no_weight_odd}.%
	\footnote{Theorem~\ref{thm:3wt_thm5_no_weight_odd} refers to the code in Lemma~\ref{lem:3wt_k3}.
	It is isomorphic to the code $C^\complement$ considered in the proof of Theorem~\ref{thm:K_N_bounds} in the smallest case $a = 1$.}

	The case $a = 2$:
	Using $K(2) \leq 17$ from Theorem~\ref{thm:K_N_bounds} and $t = 2^a = 4$ from Lemma~\ref{lem:t_divisibility}, we determined all feasible parameters computationally.
	The ones with $k \geq 7$ are listed below.
	\[
		\begin{array}{rrrrr}
			n   & k & w_1 & w_2 & w_3 \\
			\hline
			244 & 8 & 120 & 122 & 124 \\
			116 & 7 &  56 &  58 &  60 \\
			112 & 7 &  54 &  56 &  58 \\
			 16 & 7 &   6 &   8 &  10
		\end{array}
	\]
	The last code has already been excluded in Section~\ref{sec_feasible_parameters} via Lemma~\ref{lem:div_one_more}.
	Of the first three codes, the anticodes would have the parameters $[11,8,4]$, $[11,7,4]$ and $[15,7,6]$.
	The application of the Hamming bound to the punctured parameters $[10,7,3]$, $[10,6,3]$ and $[14,7,5]$ shows that these codes do not exist.

	Among all feasible parameters with $k \leq 6$, the ones with the largest possible length are $n = 56$, $k = 6$, $w_1 = 26$, $w_2 = 28$ and $w_3 = 30$.
	These parameters are realized by the anticode of the binary $[7,6]_2$ parity check code.

	The case $a = 3$:
	Similar as before, based on $K(3) \leq 25$ and $t = 8$ we were able to determine all feasible parameters computationally.
	There is only a single feasible parameter set with $k\geq 12$, which is $n=4040$, $k=12$, $w_1 = 2016$, $w_2 = 2020$ and $w_3 = 2024$.
	The anticode would have the parameters $[55, 12, 24]$, which does not exist according to the online tables \cite{grassl2007bounds}.

	Among all parameters with $k \leq 11$, the ones with the largest possible length are $n = 2024$, $k = 11$, $w_1 = 1008$, $w_2 = 1012$ and $w_3 = 1016$.
	These parameters are realized by the anticode $C^\complement$ of the binary $[23,12]_2$ Golay code $C$.
	The code $C$ has the weight enumerator $1 + 506 x^8 + 1288 x^{12} + 253 x^{16}$.
\end{proof}

Theorem~\ref{them:K_N_a_small} indicates that in general, neither the lower nor the upper bound of Theorem~\ref{thm:K_N_bounds} are sharp.
We leave it as a research problem to improve the bounds and further investigate the asymptotic growth.

\section*{Data Availability Statement}
The datasets generated during and/or analysed during the current study are available from the corresponding author on reasonable request.


\begin{thebibliography}{10}

\bibitem{triply-Munemasa}
K.~Betsumiya and A.~Munemasa, \emph{On triply even binary codes}, Journal of
  the London Mathematical Society \textbf{86} (2012), no.~1, 1--16.

\bibitem{Bonisoli1984}
A.~{Bonisoli}, \emph{{Every equidistant linear code is a sequence of dual
  Hamming codes.}}, {Ars Comb.} \textbf{18} (1984), 181--186.

\bibitem{Magma}
W.~Bosma, J.~Cannon, and C.~Playoust, \emph{The {M}agma algebra system. {I}.
  {T}he user language}, vol.~24, 1997, Computational algebra and number theory
  (London, 1993), pp.~235--265.

\bibitem{bouyukliev2021computer}
I.~Bouyukliev, S.~Bouyuklieva, and S.~Kurz, \emph{Computer
  classification of linear codes}, IEEE Transactions on Information Theory
  \textbf{67} (2021), no.~12, 7807--7814.

\bibitem{bkw2005}
M.~Braun, A.~Kohnert, and A.~Wassermann, \emph{Optimal linear codes from matrix
  groups}, IEEE Trans. Inform. Theory \textbf{51} (2005), no.~12, 4247--4251.

\bibitem{brouwerSRG}
A.~E. Brouwer, \emph{Parameters of strongly regular graphs},
  \url{https://www.win.tue.nl/~aeb/graphs/srg/srgtab.html}.

\bibitem{brouwer1989distance}
A.E. Brouwer, A.M. Cohen, and A.~Neumaier, \emph{Distance-regular graphs},
  Springer-Verlag, Berlin, 1989.

\bibitem{brouwer2011spectra}
A.E. Brouwer and W.H. Haemers, \emph{Spectra of graphs}, Springer Science \&
  Business Media, 2012.

\bibitem{BPS2016}
N.~Bruin, B.~Poonen, and M.~Stoll, \emph{Generalized explicit descent and its
  application to curves of genus 3}, Forum of Mathematics. Sigma \textbf{4}
  (2016), e6, 80pp.

\bibitem{BruinStoll2010}
N.~Bruin and M.~Stoll, \emph{The {M}ordell-{W}eil sieve: proving non-existence
  of rational points on curves}, LMS Journal of Computation and Mathematics
  \textbf{13} (2010), 272--306.

\bibitem{calderbank1986geometry}
R.~Calderbank and W.M. Kantor, \emph{The geometry of two-weight codes},
  Bulletin of the London Mathematical Society \textbf{18} (1986), no.~2,
  97--122.

\bibitem{twoweighttables}
E.~Chen, \emph{Two-weight codes},
  \url{http://moodle.tec.hkr.se/~chen/research/2-weight-codes}.

\bibitem{courteau1984triple}
B.~Courteau and J.~Wolfmann, \emph{On triple-sum-sets and two or three weights
  codes}, Discrete Mathematics \textbf{50} (1984), 179--191.

\bibitem{delsarte1972weights}
Ph. Delsarte, \emph{Weights of linear codes and strongly regular normed
  spaces}, Discrete Mathematics \textbf{3} (1972), no.~1--3, 47--64.

\bibitem{dodunekov1999some}
S.~Dodunekov, S.~Guritman, and J.~Simonis, \emph{Some new results on the
  minimum length of binary linear codes of dimension nine}, IEEE Transactions
  on Information Theory \textbf{45} (1999), no.~7, 2543--2546.

\bibitem{Dodunekov-Simonis-1998-ElecJComb5:R37}
Stefan Dodunekov and Juriaan Simonis, \emph{Codes and projective multisets},
  Electron. J. Combin. \textbf{5} (1998), \#R37, 23 pp.

\bibitem{grassl2007bounds}
M.~Grassl, \emph{Bounds on the minimum distance of linear codes and quantum
  codes}, 2007, online available at http://www.codetables.de.

\bibitem{ubt_eref40887}
D.~Heinlein, T.~Honold, M.~Kiermaier, S.~Kurz, and A.~Wassermann,
  \emph{Projective divisible binary codes}, The Tenth International Workshop on
  Coding and Cryptography 2017 : WCC Proceedings, Saint-Petersburg, September
  2017.

\bibitem{jaffe1997sextic}
D.B. Jaffe and D.~Ruberman, \emph{A sextic surface cannot have $66$ nodes},
  Journal of Algebraic Geometry \textbf{6} (1997), no.~1, 151--168.

\bibitem{kiermaier2019three}
M.~Kiermaier, S.~Kurz, M.~Shi, and P.~Sol{\'e}, \emph{Three-weight codes over
  rings and strongly walk regular graphs}, arXiv preprint 1912.03892 (2019), 28
  pp.

\bibitem{Kiermaier-Kurz-2021-arXiv:2011.05872}
M.~Kiermaier and S.~Kurz, \emph{Classification of {$\Delta$}-divisible
  linear codes spanned by codewords of weight {$\Delta$}}, arXiv preprint
  2011.05872 (2021), 28 pp.

\bibitem{kohnert2007constructing}
A.~Kohnert, \emph{Constructing two-weight codes with prescribed groups of
  automorphisms}, Discrete Applied Mathematics \textbf{155} (2007), no.~11,
  1451--1457.

\bibitem{kurz2020classification}
S.~Kurz, \emph{Classification of 8-divisible binary linear codes with
  minimum distance 24}, arXiv preprint 2012.06163 (2020), 53 pp.

\bibitem{Liu-2010-IntJInfCodingTheory1[4]:355-370}
X.~Liu, \emph{Binary divisible codes of maximum dimension}, Int. J. Inf.
  Coding Theory \textbf{1} (2010), no.~4, 355--370.

\bibitem{LMFDB:10368.j1}
The {LMFDB Collaboration}, \emph{The {L}-functions and modular forms database,
  home page of the l-function $l(s,e)$ for elliptic curve with lmfdb label
  \texttt{10368.j1}},
  \mbox{\url{https://www.lmfdb.org/EllipticCurve/Q/10368/j/1}}, 2020, [Online;
  accessed 20 July 2020].

\bibitem{lmfdb:50.a3}
\bysame, \emph{The {L}-functions and modular forms database, home page of the
  l-function $l(s,e)$ for elliptic curve with lmfdb label \texttt{50.a3}},
  \mbox{\url{https://www.lmfdb.org/EllipticCurve/Q/50/a/3}}, 2020, [Online;
  accessed 20 July 2020].

\bibitem{MacWilliams-1963-BSTJ42[1]:79-94}
J.~MacWilliams, \emph{A theorem on the distribution of weights in a
  systematic code}, Bell System Tech. J. \textbf{42} (1963), no.~1, 79--94.

\bibitem{pless1963power}
V.~Pless, \emph{Power moment identities on weight distributions in error
  correcting codes}, Information and Control \textbf{6} (1963), no.~2,
  147--152.

\bibitem{shi2019three}
M.~Shi and P.~Sol{\'e}, \emph{Three-weight codes, triple sum sets, and strongly
  walk regular graphs}, Designs, Codes and Cryptography \textbf{87} (2019),
  2395--2404.

\bibitem{simonis1994restrictions}
J.~Simonis, \emph{Restrictions on the weight distribution of binary linear
  codes imposed by the structure of {R}eed-{M}uller codes}, IEEE Transactions
  on Information Theory \textbf{40} (1994), no.~1, 194--196.

\bibitem{Stoll2001}
M.~Stoll, \emph{Implementing 2-descent for {J}acobians of hyperelliptic
  curves}, Acta Arithmetica \textbf{98} (2001), no.~3, 245--277.

\bibitem{van2013strongly}
E.R. van Dam and G.R. Omidi, \emph{Strongly walk-regular graphs}, Journal of
  Combinatorial Theory, Series A \textbf{120} (2013), no.~4, 803--810.

\bibitem{ward1981divisible}
H.N. Ward, \emph{Divisible codes}, Archiv der Mathematik \textbf{36} (1981),
  no.~1, 485--494.

\end{thebibliography}

\providecommand{\bysame}{\leavevmode\hbox to3em{\hrulefill}\thinspace}
\providecommand{\MR}{\relax\ifhmode\unskip\space\fi MR }
\providecommand{\MRhref}[2]{%
  \href{http://www.ams.org/mathscinet-getitem?mr=#1}{#2}
}
\providecommand{\href}[2]{#2}

\newpage
\appendix

\section{Generator matrix of projective three-weights codes satisfying $\mathbf{w_1+w_2+w_3=3(q-1)n/q}$}
\label{appendix_generator_matrices}

In this appendix we list examples of generator matrices corresponding to the feasible cases listed in Section~\ref{sec_feasible_parameters}.
\begin{itemize}
  \item $q=2$, $n=4$, $k=3$, $w=[1, 2, 3]$:
        $\left(\begin{smallmatrix}
          1000\\
          0101\\
          0011\\
        \end{smallmatrix}\right)$
  \item $q=2$, $n=8$, $k=4$, $w=[2, 4, 6]$:
        $\left(\begin{smallmatrix}
          01111011\\
          01101010\\
          10101100\\
          10110010\\
        \end{smallmatrix}\right)$
  \item $q=2$, $n=8$, $k=5$, $w=[2, 4, 6]$:
        $\left(\begin{smallmatrix}
11101110\\
01010000\\
00111010\\
10001000\\
11000011\\
        \end{smallmatrix}\right)$

  \item $q=2$, $n=8$, $k=6$, $w=[2, 4, 6]$:
        $\left(\begin{smallmatrix}
00110110\\
00010001\\
01010011\\
10010110\\
10100110\\
01111101\\
       \end{smallmatrix}\right)$

\item $q=2$, $n=12$, $k=5$, $w=[4, 6, 8]$:
       $\left(\begin{smallmatrix}
       1 0 0 1 0 0 1 1 1 0 0 1\\
       0 1 0 1 0 0 1 1 1 1 0 0\\
       0 0 1 0 0 0 1 1 1 1 0 1\\
       0 0 0 0 1 0 1 1 0 0 1 0\\
       0 0 0 0 0 1 1 0 1 0 1 0\\
       \end{smallmatrix}\right)$

\item $q=2$, $n=12$, $k=6$, $w=[4, 6, 8]$:
       $\left(\begin{smallmatrix}
       1 0 0 0 0 0 0 0 0 1 1 1\\
       0 1 0 0 0 0 1 1 0 0 1 0\\
       0 0 1 0 0 0 1 1 0 1 0 0\\
       0 0 0 1 0 0 1 1 1 0 1 1\\
       0 0 0 0 1 0 0 0 1 0 1 1\\
       0 0 0 0 0 1 1 0 0 0 1 1\\
       \end{smallmatrix}\right)$

\item $q=2$, $n=16$, $k=5$, $w=[6, 8, 10]$:
       $\left(\begin{smallmatrix}
1 0 1 1 1 1 1 0 0 0 1 1 0 0 0 0\\
1 1 1 1 0 1 1 0 0 1 0 1 1 0 0 1\\
0 1 1 1 1 1 1 1 1 0 1 0 1 0 0 0\\
0 1 1 0 0 1 0 1 0 0 1 1 1 1 1 1\\
1 0 0 0 0 1 1 1 1 1 1 0 1 0 1 1\\
       \end{smallmatrix}\right)$
\item $q=2$, $n=16$, $k=6$, $w=[6, 8, 10]$:
       $\left(\begin{smallmatrix}
0 0 0 0 1 0 0 1 1 0 0 1 0 0 1 1\\
0 0 1 0 1 1 0 0 0 1 1 0 1 0 0 0\\
0 0 0 0 0 1 1 0 1 1 0 0 1 0 0 1\\
1 0 0 1 0 1 0 0 0 0 1 1 0 1 0 0\\
0 0 0 1 0 0 1 1 0 0 1 0 0 1 0 1\\
0 1 0 1 1 0 0 0 0 1 0 1 0 0 1 0\\
       \end{smallmatrix}\right)$
\item $q=2$, $n=16$, $k=5$, $w=[4, 8, 12]$:
       $\left(\begin{smallmatrix}
1 1 0 0 0 1 1 1 0 1 1 0 0 1 0 0\\
1 1 0 0 1 1 0 0 0 1 0 0 1 0 1 1\\
1 0 1 0 1 0 1 0 0 1 1 1 0 0 1 0\\
1 1 0 1 1 0 0 0 0 1 1 0 1 0 0 1\\
0 1 1 1 1 1 1 0 1 1 1 0 0 1 1 1\\
       \end{smallmatrix}\right)$
\item $q=2$, $n=16$, $k=6$, $w=[4, 8, 12]$:
       $\left(\begin{smallmatrix}
1 1 0 1 0 0 1 0 1 0 0 1 0 1 0 1\\
1 1 1 0 0 0 1 0 0 0 0 1 1 1 0 1\\
1 1 0 1 0 0 0 1 0 1 0 0 1 0 1 1\\
1 0 0 1 0 1 1 1 1 1 1 0 0 0 0 0\\
1 1 0 0 1 1 1 0 1 1 1 1 1 0 1 1\\
1 1 1 0 0 1 1 1 0 0 0 1 1 0 0 0\\
       \end{smallmatrix}\right)$
\item $q=2$, $n=16$, $k=7$, $w=[4, 8, 12]$:
       $\left(\begin{smallmatrix}
1 0 0 0 0 0 0 1 1 1 0 1 1 1 1 0\\
0 1 0 0 0 0 0 1 0 1 1 0 0 0 0 0\\
0 0 1 0 0 0 0 0 0 1 1 1 0 0 0 0\\
0 0 0 1 0 0 0 0 0 1 1 0 1 0 0 0\\
0 0 0 0 1 0 0 0 0 1 1 0 0 1 0 0\\
0 0 0 0 0 1 0 1 1 1 0 1 1 1 0 1\\
0 0 0 0 0 0 1 1 0 1 0 1 1 1 1 1\\
       \end{smallmatrix}\right)$

\item $q=2$, $n=20$, $k=5$, $w=[8, 10, 12]$:
       $\left(\begin{smallmatrix}
11101110100010010100\\
00101001011100000110\\
10100101000111011001\\
11101000100110100011\\
11010101010001000001\\
       \end{smallmatrix}\right)$

\item $q=2$, $n=24$, $k=5$, $w=[10, 12, 14]$:
       $\left(\begin{smallmatrix}
111011100010100000010111\\
011010110110000111101011\\
000011011110001101010110\\
010110011101101000001101\\
100000101111110110010001\\
       \end{smallmatrix}\right)$
\item $q=2$, $n=24$, $k=6$, $w=[8, 12, 16]$:
       $\left(\begin{smallmatrix}
1 0 0 0 0 0 0 1 0 1 1 1 1 0 0 1 0 0 0 1 1 1 1 1\\
0 1 0 0 0 1 0 1 0 0 1 1 0 1 1 1 1 1 0 0 1 1 0 0\\
0 0 1 0 0 0 0 0 1 0 1 1 0 1 1 0 1 0 1 1 0 1 1 1\\
0 0 0 1 0 1 0 1 1 1 0 1 0 1 0 0 0 1 1 1 1 0 1 0\\
0 0 0 0 1 1 0 0 0 1 0 0 1 0 1 0 0 0 1 1 0 0 0 1\\
0 0 0 0 0 0 1 0 1 0 0 0 0 1 1 0 0 0 0 0 1 1 1 1\\
       \end{smallmatrix}\right)$
\item $q=2$, $n=24$, $k=7$, $w=[8, 12, 16]$:
       $\left(\begin{smallmatrix}
001111111011101110110001\\
010111110111011101101010\\
010110111111111010101001\\
100111101110111110011001\\
001111011101111101110010\\
100101111111110101011010\\
000000000011111111111100\\
       \end{smallmatrix}\right)$
\item $q=2$, $n=24$, $k=8$, $w=[8, 12, 16]$:
       $\left(\begin{smallmatrix}
100011011110010011111111\\
011100111001100010001101\\
001101100110110000010111\\
001010101101010101100101\\
011101110001000111010001\\
110111011011101100101110\\
010011010110010111010001\\
001111011111101100111001\\
       \end{smallmatrix}\right)$
\item $q=2$, $n=24$, $k=9$, $w=[8, 12, 16]$:
       $\left(\begin{smallmatrix}
011110011110001111001111\\
111101111000010111110101\\
001011101100101000011011\\
111110100101000000101101\\
001100110000101011110101\\
010101000001110001111011\\
001101010110000110101011\\
110000101110101110100001\\
001111111100001110110111\\
       \end{smallmatrix}\right)$
\item $q=2$, $n=24$, $k=10$, $w=[8, 12, 16]$:
       $\left(\begin{smallmatrix}
    1 0 0 0 0 0 0 0 0 0 0 1 1 0 1 0 1 0 0 0 1 0 1 1\\
    0 1 0 0 0 0 0 0 0 1 0 1 0 0 1 0 0 1 0 0 0 1 1 1\\
    0 0 1 0 0 0 0 0 0 0 0 1 1 1 1 0 0 0 0 1 0 1 0 1\\
    0 0 0 1 0 0 0 0 0 0 0 0 0 1 1 1 0 0 1 1 1 1 0 0\\
    0 0 0 0 1 0 0 0 0 1 0 1 0 0 1 1 0 0 1 0 1 0 1 0\\
    0 0 0 0 0 1 0 0 0 0 0 1 0 0 1 0 0 1 1 1 1 0 0 1\\
    0 0 0 0 0 0 1 0 0 0 0 1 0 0 1 0 1 0 1 1 0 1 1 0\\
    0 0 0 0 0 0 0 1 0 1 0 1 0 0 0 1 0 0 1 1 0 1 0 1\\
    0 0 0 0 0 0 0 0 1 1 0 1 0 0 0 1 1 1 0 0 1 0 0 1\\
    0 0 0 0 0 0 0 0 0 0 1 0 0 0 0 1 1 0 1 1 1 0 1 1\\
       \end{smallmatrix}\right)$
\item $q=2$, $n=24$, $k=11$, $w=[8, 12, 16] $:
       $\left(\begin{smallmatrix}
1 0 0 0 0 0 0 0 0 0 0 0 0 1 1 0 1 1 0 1 1 0 1 0\\
0 1 0 0 0 0 0 0 0 0 0 0 0 1 0 1 1 1 0 0 0 1 1 1\\
0 0 1 0 0 0 0 0 0 0 0 0 1 1 0 0 1 0 0 1 1 1 0 1\\
0 0 0 1 0 0 0 0 0 0 0 0 1 1 0 0 0 1 1 1 0 0 1 1\\
0 0 0 0 1 0 0 0 0 0 1 0 0 1 1 0 0 1 0 1 0 1 0 1\\
0 0 0 0 0 1 0 0 0 0 1 0 0 1 1 1 1 0 0 0 1 0 0 1\\
0 0 0 0 0 0 1 0 0 0 0 0 0 1 1 0 0 0 1 0 1 1 1 1\\
0 0 0 0 0 0 0 1 0 0 1 0 0 0 1 0 1 1 1 0 0 0 1 1\\
0 0 0 0 0 0 0 0 1 0 1 0 0 0 1 0 1 0 1 1 1 1 0 0\\
0 0 0 0 0 0 0 0 0 1 0 0 0 0 1 1 0 1 1 1 0 1 1 0\\
0 0 0 0 0 0 0 0 0 0 0 1 0 0 0 1 1 0 1 1 1 0 1 1\\
       \end{smallmatrix}\right)$

\item $q=2$, $n=32$, $k=6$, $w=[12, 16, 20]$:
       $\left(\begin{smallmatrix}
100100001101011011001 11100111000\\
100110011111001001110 00010100011\\
111101110001110001101 01101110110\\
000011111111111100011 11000000000\\
101010100001011110110 01110100100\\
000001001001001011011 11011101101\\
       \end{smallmatrix}\right)$
\item $q=2$, $n=32$, $k=7$, $w=[12, 16, 20]$:
       $\left(\begin{smallmatrix}
01111000000111110000111111111110\\
11100110110010000111111000110000\\
10011111001000101101101100100100\\
10110110011001001011011100101000\\
01111100000000001111111111011111\\
11001111100100001110110100100010\\
01111110000000000000000001111110\\
       \end{smallmatrix}\right)$
\item $q=2$, $n=32$, $k=8$, $w=[12, 16, 20]$:
       $\left(\begin{smallmatrix}
11101100110111010010100111110011\\
10011000100011110110110010001101\\
11011101101110100001011110100111\\
00110101000110101101110100011001\\
10111111001101000110101101001111\\
01101010001100011111101000110001\\
01111110011011001001011011011011\\
11000100011001111011010001100101\\
       \end{smallmatrix}\right)$
\item $q=2$, $n=32$, $k=9$, $w=[12, 16, 20]$:
       $\left(\begin{smallmatrix}
1 0 0 0 0 0 0 0 0 0 1 0 1 1 0 1 1 0 0 1 1 1 1 0 0 0 0 0 1 0 1 0\\
0 1 0 0 0 0 0 0 0 1 0 0 0 0 0 1 1 1 0 1 1 0 0 0 0 1 1 0 1 1 1 0\\
0 0 1 0 0 0 0 0 0 0 0 0 1 0 0 1 0 1 0 1 0 1 1 0 1 1 0 0 1 1 0 1\\
0 0 0 1 0 0 0 0 0 1 0 0 1 0 1 0 1 0 1 1 1 0 1 1 1 1 1 1 0 1 1 0\\
0 0 0 0 1 0 0 0 0 0 0 0 0 1 0 1 1 0 1 1 0 0 1 0 1 1 1 0 0 0 1 1\\
0 0 0 0 0 1 0 0 0 1 1 1 1 1 0 0 1 1 0 1 0 0 0 0 0 0 1 1 1 0 0 0\\
0 0 0 0 0 0 1 0 0 1 1 0 0 0 1 1 0 0 0 0 0 0 1 1 1 0 1 1 0 0 1 1\\
0 0 0 0 0 0 0 1 0 1 1 0 1 1 0 0 0 0 1 0 0 0 1 1 0 0 0 1 1 1 1 0\\
0 0 0 0 0 0 0 0 1 0 0 1 1 0 0 1 1 1 1 0 1 0 1 0 0 0 1 0 1 0 0 1\\
       \end{smallmatrix}\right)$
\item $q=2$, $n=32$, $k=10$, $w=[12, 16, 20]$:
       $\left(\begin{smallmatrix}
11111000001111110000001000000000\\
00000111111111110000000100000000\\
00011000110000111111100010000000\\
01101000010011010011110001000000\\
10110011011101000100100000100000\\
01100101110101001000110000010000\\
11010010110111000001010000001000\\
11100011011010100011000000000100\\
10101010110011101000100000000010\\
10011001000011101011010000000001\\
       \end{smallmatrix}\right)$
\item $q=2$, $n=32$, $k=6$, $w=[8, 16, 24]$:
       $\left(\begin{smallmatrix}
01111011100101011111111101101111\\
01001000011110111001110001101001\\
01010011011000001100111010110011\\
10011110100001010010011010110110\\
01000110101011010100110101011010\\
10011000101110001110100111100010\\
       \end{smallmatrix}\right)$
\item $q=2$, $n=32$, $k=7$, $w=[8, 16, 24]$:
       $\left(\begin{smallmatrix}
11010100001101001100010111100101\\
01010011101101010011000101010110\\
11101110100111100010000010110100\\
11100110010010111100111001100000\\
11111101000001001111100001011000\\
01000011101110010000011101110110\\
10001000111111010101110011010000\\
       \end{smallmatrix}\right)$
\item $q=2$, $n=32$, $k=8$, $w=[8, 16, 24]$:
       $\left(\begin{smallmatrix}
1 0 0 0 0 0 0 1 0 1 1 1 1 1 1 0 1 0 0 0 0 0 0 1 0 1 1 1 1 1 1 0\\
0 1 0 0 0 0 0 1 0 1 0 0 0 0 0 1 0 1 0 0 0 0 0 1 0 1 0 0 0 0 0 1\\
0 0 1 0 0 0 0 1 0 0 1 0 0 0 0 1 0 0 1 0 0 0 0 1 0 0 1 0 0 0 0 1\\
0 0 0 1 0 0 0 1 0 0 0 1 0 0 0 1 1 1 1 0 1 1 1 0 1 1 1 0 1 1 1 0\\
0 0 0 0 1 0 0 1 0 0 0 0 1 0 0 1 1 1 1 1 0 1 1 0 1 1 1 1 0 1 1 0\\
0 0 0 0 0 1 0 1 0 0 0 0 0 1 0 1 1 1 1 1 1 0 1 0 1 1 1 1 1 0 1 0\\
0 0 0 0 0 0 1 1 0 0 0 0 0 0 1 1 1 1 1 1 1 1 0 0 1 1 1 1 1 1 0 0\\
0 0 0 0 0 0 0 0 1 1 1 1 1 1 1 1 1 1 1 1 1 1 1 1 0 0 0 0 0 0 0 0\\
       \end{smallmatrix}\right)$
\item $q=2$, $n=32$, $k=9$, $w=[8, 16, 24]$:
       $\left(\begin{smallmatrix}
10101110000101001100011100001111\\
01010101001101110001010100101011\\
10100010100110111000011100001111\\
00110000101001101110100101101101\\
11001111111111111110110110100101\\
01101111010110010000101010101010\\
10010000000000000001100010001110\\
01100000000000000001000110011100\\
01011000010100110111010100101011\\
       \end{smallmatrix}\right)$

\item $q=2$, $n=40$, $k=6$, $w=[16, 20, 24]$:\\
       $\left(\begin{smallmatrix}
0100111000101111011110010010110000010011\\
1010010100011111101111001001011000001001\\
1101001010000111110111101100100100000101\\
0110100101001011110011110110010010000011\\
0011010010101101111001111011000001001001\\
1001100001011110111100110101100000100101\\
       \end{smallmatrix}\right)$
\item $q=2$, $n=40$, $k=7$, $w=[16, 20, 24]$:\\
       $\left(\begin{smallmatrix}
1000011111010101011111010011111001101010\\
0000100101101011101110100000011011011101\\
0100000011110101110110010001001100101111\\
0001001001011101111101000000110100111011\\
1000011110101011011110101101111001110100\\
0010010000111011111010001001100011010111\\
0111100000000111011110000000000111111111\\
       \end{smallmatrix}\right)$
\item $q=2$, $n=40$, $k=8$, $w=[16, 20, 24]$:\\
       $\left(\begin{smallmatrix}
1111100000011111011111000000111111010110\\
0110011000101000110100111101011011000011\\
1100010001110000101001111011110101000011\\
0001100110100010100101101111010111000011\\
0000011111000000111111000001000001011010\\
1000100011100001110010110111101011000011\\
0011001100100100101010011111101101000011\\
1111100000011111000000000001000001111001\\
       \end{smallmatrix}\right)$
\item $q=2$, $n=40$, $k=9$, $w=[16, 20, 24]$:\\
       $\left(\begin{smallmatrix}
0100010100011110100101111111111011001101\\
1101110000010010111011001100011011000101\\
0001001100100111100011111100111011111110\\
0000011101111101110100001100101000110110\\
1101001011110010001000001010000011100001\\
1000101000100000100110010000010010111111\\
1111110100101101100101001001011111010101\\
0010100010101001000010011111000010110100\\
0110010000100111001101001010010010110000\\
       \end{smallmatrix}\right)$

\item $q=2$, $n=48$, $k=6$, $w=[22, 24, 26]$:\\
       $\left(\begin{smallmatrix}
100100010010111011001011111101000110011100111001\\
100010001101011101100101111110000011001110111100\\
100001010110101010110110111011100001000111011110\\
100000101111010001011111011101110000100011001111\\
110000010011101100101111101110011000110001100111\\
101000001101110110010111110011001100111000110011\\
       \end{smallmatrix}\right)$
\item $q=2$, $n=48$, $k=6$, $w=[20, 24, 28]$:\\
       $\left(\begin{smallmatrix}
111011100011101101000010100000100001101111100110\\
001011011110010000100111110011000100111011100100\\
111110001110011011000100001000001010011111001101\\
010110110101100000001111101110000001111110001001\\
111101010101110110000001010000010100110111010011\\
100101101011001000010111011101000010111101010010\\
       \end{smallmatrix}\right)$
\item $q=2$, $n=48$, $k=7$, $w=[20, 24, 28]$:\\
       $\left(\begin{smallmatrix}
100010111101100101011111000000010001010100111101\\
001010110111010101110100011001000101000111110001\\
000101111011001010111110001000100010100101111010\\
010100101111100011101000111010000010001111100110\\
010001111110010110101011100000001000101110011110\\
101000011111001111010001110100000100010111001101\\
000000000000111000000111111111111111111111111100\\
       \end{smallmatrix}\right)$
\item $q=2$, $n=48$, $k=8$, $w=[20, 24, 28]$:\\
$\left(\begin{smallmatrix}
111111111000000000111111111100000000000001000000\\
000000000111111111111111111100000000000000100000\\
000011111000011111000001111111111111000000010000\\
001100111001100111001110001100001111111100001000\\
010101001010101001010000110101110011000110000100\\
100101010111111110010010011010010101011010000010\\
111101010001111011100110111000100110101000000001\\
101001000000010011001010111001011011011110000000\\
       \end{smallmatrix}\right)$
\item $q=2$, $n=48$, $k=7$, $w=[16, 24, 32]$:\\
       $\left(\begin{smallmatrix}
100101110010111001011100101110010111001011000000\\
110100000110111111100011001101011010010001010100\\
100111000010110000001001010101111011101100001111\\
110101101001000111010000011011111110001100010001\\
111010000011011111110001100110101101001000001010\\
100010000111110101001100010000111110101001101101\\
111100111100111100111100111100111100111100011110\\
       \end{smallmatrix}\right)$
\item $q=2$, $n=48$, $k=8$, $w=[16, 24, 32]$:\\
       $\left(\begin{smallmatrix}
100111101111101111110001110001110111101011101100\\
101000110000110000010010010010001000101100110000\\
001010011001011001001110001110001010001011101111\\
001010101010010101111101000010100101111010011001\\
100111111010111010100100100100000111101011100000\\
111001010001010001110110110110001001110111010000\\
001010011001011001110001110001111010001011101100\\
001111010011101100010000101111100000110110101001\\
       \end{smallmatrix}\right)$
\item $q=2$, $n=48$, $k=9$, $w=[16, 24, 32]$:\\
       $\left(\begin{smallmatrix}
110110000011101101100000111010100001011000101111\\
100111011100100110001000110110010001101111001010\\
110000000011110011111111000001101000011011110010\\
100100011100101001000111001011111101001110100100\\
011001010101101001101010100100111101111001000001\\
001100011110100011000111101010100001011000101111\\
110011001100110011001100110010010101010101010110\\
101011011100011010110111000100000011101001110100\\
111011101011111110111010111101011100010110001011\\
       \end{smallmatrix}\right)$
\item $q=2$, $n=48$, $k=10$, $w=[16, 24, 32]$:\\
       $\left(\begin{smallmatrix}
100100110000001011001110110100110001111101101100\\
000110000011010001000000101110111111110011100111\\
110001010110011100100101011100100101011011000101\\
000001101100110001101010001110010101001111111001\\
011100100101100100111010100100111010010101110010\\
111110000111001100010010001100010010011111111000\\
111111100001100011000100100011000100000111111110\\
011110011011101000111001010111000110010010000110\\
001011000001111001011011111001011011000100101100\\
101010011001010111101011101000010100011001010110\\
       \end{smallmatrix}\right)$
\item $q=2$, $n=48$, $k=11$, $w=[16, 24, 32]$:\\
       $\left(\begin{smallmatrix}
100110100101001011010101010101011010100110100110\\
000110110011001011010101010101011010011001011001\\
000000001111000000000000000000111111111111000000\\
010011100000000000000011000011001111000000110011\\
001111000000000110000011000000001111110011000000\\
000000000000100111000000000000110011110011111100\\
000000000000011110000000000000001111111100110011\\
000000000000000000110011000000111100110011110011\\
000000000000000000001111000000001111001111111100\\
000000000000000000000000110011110011001111110011\\
000000000000000000000000001111001111110011001111\\
       \end{smallmatrix}\right)$
\item $q=2$, $n=48$, $k=12$, $w=[16, 24, 32]$:\\
       $\left(\begin{smallmatrix}
100101000010110010101011001011001101001010110111\\
010011000010110010101011001101010010110010101111\\
001111000000000000000001100110000111100111100000\\
000000100001100000000110000001111000000111111001\\
000000010001100000000111100111100110000001100010\\
000000001001110000000111100110011111111111111000\\
000000000111100000000110000001111111111000000000\\
000000000000001000000110000111100111100110011100\\
000000000000000110000111100111111000011000011000\\
000000000000000001100111100001100001111111100000\\
000000000000000000011001100001111111100001111000\\
000000000000000000000000011111111110011111100000\\
       \end{smallmatrix}\right)$

\item $q=2$, $n=52$, $k=6$, $w=[24, 26, 28]$:\\
       $\left(\begin{smallmatrix}
0110011010101010101101010010101001010101010101100000\\
1110110110000110011011001001100100110011001100010000\\
1101110001100001111000111000011100001111000011001000\\
0011110000011111111000000111111100000000111111000100\\
0000001111111111111000000000000011111111111111000010\\
0000000000000000000111111111111111111111111111000001\\
       \end{smallmatrix}\right)$

\item $q=2$, $n=56$, $k=6$, $w=[26, 28, 30]$:\\
       $\left(\begin{smallmatrix}
01011001100101011111011100100001001101100011111010000100\\
10001111101110110100110100101000001001111001001111101010\\
01110111001001000011000111011111010001111001011111010000\\
10011111000011000111100111111000111100000111000100100001\\
11011010110001000100011000011100010110111101010010111111\\
00000110101101000110100101000110011111110000110000100111\\
       \end{smallmatrix}\right)$
\item $q=2$, $n=56$, $k=7$, $w=[24, 28, 32]$:\\
       $\left(\begin{smallmatrix}
11001100110101111111011001000101111101000001100011011101\\
01100111011010011111101100101010111010100010110001101110\\
10110010101101101111100110011101011001010001011000110111\\
11011001010110110111110011001110101000101000101101011011\\
01101100101011111011101001101111010000010100010111101101\\
00110111010101111101100100110111101100001010001011110110\\
10011011101010111110110010011011110010000111000100111011\\
       \end{smallmatrix}\right)$
\item $q=2$, $n=56$, $k=8$, $w=[24, 28, 32]$:\\
       $\left(\begin{smallmatrix}
00111011101011010100101010111101100110111001010100001111\\
01110110010110111001010001111011001101110010101100011110\\
11101100101101100010100111110110011011100101011000111100\\
11011001011011010101001011101101110111001010110001111000\\
10110011110110101010010011011011101110010101100111110000\\
01100111101101010100100110110111011100111011001011100001\\
11001110011010111001001001101111111001100110010111000011\\
10011101110101100010010111011110110011011100101010000111\\
       \end{smallmatrix}\right)$
\item $q=2$, $n=56$, $k=9$, $w=[24, 28, 32]$:\\
       $\left(\begin{smallmatrix}
1 0 0 0 0 0 0 0 0 1 1 1 1 1 0 0 0 0 0 0 1 1 1 1 1 1 1 1 0 0 0 0 0 0 1 1 1 1 1 1 1 1 0 0 0 0 1 1 1 1 1 1 1 1 1 1\\
0 1 0 0 0 0 1 0 0 1 1 1 0 0 0 1 0 0 0 0 0 1 0 1 1 0 0 1 0 1 0 0 1 1 0 1 1 1 0 1 1 1 1 0 0 1 0 1 0 0 0 0 1 0 1 0\\
0 0 1 0 0 0 1 0 0 1 0 1 1 0 0 1 1 1 0 0 0 0 0 0 1 0 1 0 0 1 1 1 1 1 1 0 1 0 1 0 1 0 0 1 0 1 0 0 1 0 0 0 1 1 0 0\\
0 0 0 1 0 0 1 0 0 0 0 1 1 0 0 0 0 1 1 0 1 0 1 1 1 1 1 0 0 0 0 0 1 1 0 0 1 1 1 1 0 0 0 0 1 1 1 1 0 1 0 0 1 0 0 0\\
0 0 0 0 1 0 1 0 1 0 1 0 1 0 0 0 0 1 1 0 0 1 1 0 0 0 1 0 0 1 0 1 1 1 0 1 1 0 1 1 1 0 0 1 0 0 0 1 0 0 1 1 0 0 1 0\\
0 0 0 0 0 1 1 0 1 1 1 0 1 0 0 0 1 0 1 1 1 0 1 1 0 0 0 0 0 0 0 1 1 1 1 0 0 0 0 0 0 0 0 1 1 1 0 0 0 1 1 1 1 1 0 0\\
0 0 0 0 0 0 0 1 0 0 1 1 1 1 0 1 1 1 1 0 1 1 1 0 1 1 0 0 0 0 0 1 1 0 1 0 0 1 1 0 0 0 0 1 0 0 0 0 1 1 1 0 0 0 1 0\\
0 0 0 0 0 0 0 0 0 0 0 0 0 0 1 1 1 1 1 1 1 1 1 1 1 1 1 1 0 0 0 0 0 0 0 0 0 0 0 0 0 0 1 1 1 1 1 1 1 1 1 1 1 1 1 1\\
0 0 0 0 0 0 0 0 0 0 0 0 0 0 0 0 0 0 0 0 0 0 0 0 0 0 0 0 1 1 1 1 1 1 1 1 1 1 1 1 1 1 1 1 1 1 1 1 1 1 1 1 1 1 1 1\\
       \end{smallmatrix}\right)$

\item $q=2$, $n=64$, $k=7$, $w=[28, 32, 36]$:\\
       $\left(\begin{smallmatrix}
1010111010110001100001101011011110110111100111100110101000101001\\
0011000011000101110101110110111111101100100010100011100111100111\\
0111110001001111011001100101001001100110001001100101000001110111\\
0001110111000100110010001000100110111101001100100000001110111000\\
1100000111011000010011000010101011110010000100110011100000111010\\
0000101101001100010001110111100010010100111011100001111010010111\\
1111100010001110110011011010000011001110010011001000000011101111\\
       \end{smallmatrix}\right)$
\item $q=2$, $n=64$, $k=8$, $w=[28, 32, 36]$:\\
       $\left(\begin{smallmatrix}
1000000000111101101011011100000010010011001110100110001100100011\\
0100000000100011011110110010000011011010101001110101001010110010\\
0010000001010111010010111101000111011000000111010000110000111100\\
0001000000101011101001011110100011101100000011101000011000011110\\
0000100001010011001001001011010111000011010010011110011001101010\\
0000010100001111111001101001110010100111101001000000010110000011\\
0000001101010111001100001000101001011001001110110111100011111011\\
0000000011110110101101110000001001001100111010011000110010001101\\
       \end{smallmatrix}\right)$
\item $q=2$, $n=64$, $k=9$, $w=[28, 32, 36]$:\\
       $\left(\begin{smallmatrix}
1000000000111101101011011100000010010011001110100110001100100011\\
0100000000100011011110110010000011011010101001110101001010110010\\
0010000000101100000100000101000011111110011010011100101001111011\\
0001000000101011101001011110100011101100000011101000011000011110\\
0000100000101000011111110011010011100101001111010010000000101101\\
0000010100001111111001101001110010100111101001000000010110000011\\
0000001100101100011010110000101101111111010011111011111010111100\\
0000000010001101111011001000001101101010100111010100101011001010\\
0000000001111011010110111000000100100110011101001100011001000111\\
       \end{smallmatrix}\right)$
\item $q=2$, $n=64$, $k=7$, $w=[24, 32, 40]$:\\
       $\left(\begin{smallmatrix}
0100111111000111001110101011100000010111100001010100101001011100\\
0010111101111100111000100110011001001110100100010010100101110010\\
1010010111110011100110010101110100001011110000001010110100001110\\
0101011011111001110011001000111010001101101000000101101010100110\\
1001011110101110011101010011001000101111000010001001110010111000\\
1111100000000000111111111100000111111111100000000000100000111111\\
1111100000000000000000000011111000001111111111111111100000111110\\
       \end{smallmatrix}\right)$
\item $q=2$, $n=64$, $k=8$, $w=[24, 32, 40]$:\\
       $\left(\begin{smallmatrix}
1011111101010110101001010001010100011100001111101000011000001101\\
0111111110101001010101100000101100011100001111100100010100001101\\
0011111111111111111100110001100011111110001000001101011100011101\\
1100011111111100000000001111100000011100001000111110100011110101\\
1110100110010010010000001100011101101111100011100010110010010111\\
1111010001001001001000000110011110110110110101100010110001010111\\
1101101000100100100100001010011111011011010110100010110000110111\\
0011100000000011111111001111111000000010001111100010001111100110\\
       \end{smallmatrix}\right)$
\item $q=2$, $n=64$, $k=9$, $w=[24, 32, 40]$:\\
       $\left(\begin{smallmatrix}
1000110111000011000011000010010100111011001011011110011010011011\\
0100000100001100111100111101100111110111111011010000110000110000\\
0000010000010011110011110110111111011111011101100001000011000010\\
1011001111100111101100111100001000010001011010001000001010111101\\
1011001011001000101000100110111001101001101101100001110101110100\\
1000110001110011100101010101111101111111101111000011011010110111\\
0000011101110010111111011001111101010110000101001011000010011000\\
0001101100001001100100001001110001001010010101001010001011010000\\
1111101001010011010101100001011001100100001001111100000011101110\\
       \end{smallmatrix}\right)$
\item $q=2$, $n=64$, $k=10$, $w=[24, 32, 40]$:\\
       $\left(\begin{smallmatrix}
1000000100001100100000110000100110010001010101101110000011011001\\
0100000011011100000100101000010110100011101100101010001000001001\\
0010000111011100110001101000000001001110101001100000011000010010\\
0001000011001100111010000100100110000001010110000100110101110000\\
0000100010010010000011001001100001000000011100010011010011101111\\
0000010110000010001111011010100101001111011100110111111110000001\\
0000001100000010110110011011010100010010100100000101101110000010\\
0000000000111010001100001010100000100001100010101001011001101111\\
0000000000000001111111111111110000000111111111111111100000000010\\
0000000000000000000000000000001111111111111111111111100000000001\\
       \end{smallmatrix}\right)$
\item $q=2$, $n=64$, $k=11$, $w=[24, 32, 40]$:\\
       $\left(\begin{smallmatrix}
0000011010000001000110111110001110100101111000110010010000000000\\
0111001100100010000011101000001011100100101101101010001000000000\\
0000101110101011000110100111100101110100100101000000000100000000\\
0111101110001010001100101011001010000101010011010000000010000000\\
1001110100000110100000011010000101011011000100111001100001000000\\
1101010010011110010001000110111100000011001100001000100000100000\\
0101010001111110001011000001110100001000000011000111100000010000\\
0011001111111110000111000000010010111000111110111111100000001000\\
0000111111111110000000111111110001111000000001111111100000000100\\
0000000000000001111111111111110000000111111111111111100000000010\\
0000000000000000000000000000001111111111111111111111100000000001\\
       \end{smallmatrix}\right)$
\item $q=2$, $n=64$, $k=12$, $w=[24, 32, 40]$:\\
       $\left(\begin{smallmatrix}
0000110001101110000100100100100011011000011011011110100000000000\\
1011110000100110010000001100010000111101001110111000010000000000\\
1010110001001010110010000000101111110000001100101011001000000000\\
1111100000001100000010100100111101000011011011101000000100000000\\
0111000000001010110110001100011000000110111100110011000010000000\\
0000000100001001111110011010010101001101010101010101000001000000\\
0101011111010000010001111001110011000100100000101100000000100000\\
0011010011001000001111111001111111011111010011011100000000010000\\
0000101111000110000000000111101111000011001111000011000000001000\\
0000011111000001111111111111100000111111000000111111000000000100\\
0000000000111111111111111111100000000000111111111111000000000010\\
0000000000000000000000000000011111111111111111111111000000000001\\
       \end{smallmatrix}\right)$
\item $q=2$, $n=64$, $k=7$, $w=[16, 32, 48]$:\\
       $\left(\begin{smallmatrix}
1110001110001101110111000101111010000010001011110100001100100110\\
0000101111010000111110100101001111100101101101100000110011100010\\
0010111101001001011011001011000001111100001001111100101110001001\\
0111101000011110000111000111110010110100101000011010110001110001\\
1110100001011000011101111111001010010010110001101011000111000100\\
0111000111000100010001111101000010111011101010000101111010010011\\
1111010000100011110010101111100101010110100000110101100011100010\\
       \end{smallmatrix}\right)$
\item $q=2$, $n=64$, $k=8$, $w=[16, 32, 48]$:\\
       $\left(\begin{smallmatrix}
1000110100011000110100010111001011100101111111000110111001010000\\
1101100001001100110101011110000010101100110010111001000110111001\\
1111011000010010001100101100111011110000010010101110010001111001\\
0011110110000101011101110000010101100111011010110100110100001001\\
0011110110000101011100101011001110111100001010110100110100000011\\
1000010011110110111000111011110000010101101011101000101000101001\\
1000010011110101000111010110011101111000000010010111101000110101\\
0111101100001001000111100000101011001110110100010111101000110011\\
       \end{smallmatrix}\right)$
\item $q=2$, $n=64$, $k=9$, $w=[16, 32, 48]$:\\
       $\left(\begin{smallmatrix}
0101011110000011010000110100011010001101001110111011100010011101\\
0110111110010000000011011011011011011011000101101000111101010000\\
1010101010100100011101000110100011010001101010111100101011010101\\
1010101010100001101100011010001101000110101010111100101011010101\\
0101010011100100011101000110100011010001100011111010111001000111\\
1100010100110110100010110011101111000001101111010100010110000101\\
1111110010011100101111001011100101110010010010000110100100100100\\
0001110011011000110111100000101011001110111110001101011000010001\\
0000000110110111001101110010111001011100110110000001101101101100\\
       \end{smallmatrix}\right)$
\item $q=2$, $n=64$, $k=10$, $w=[16, 32, 48]$:\\
       $\left(\begin{smallmatrix}
0010011111000000110010110100111010101001010101111111110000011000\\
1001000110100111000010000111101011001111000111011101101001100001\\
1110001110010011001001010000011111011111010001001000011011011010\\
1110000011100110100100010100000100010111110111111011000110101100\\
1101100000000001111000000000101000000000000001010000000000111100\\
0101010010100000111011001001001110101100010011101011010111100011\\
1010100110001111000111110101110001010100001100010110011100011100\\
1111111000011101001010100101010010001011010100101110000111100100\\
1111110010110110010011110000100100100001111001001100101101001001\\
1111110101101001001011100001010010000011110100101101011010100100\\
       \end{smallmatrix}\right)$
\item $q=2$, $n=64$, $k=11$, $w=[16, 32, 48]$:\\
       $\left(\begin{smallmatrix}
0100001000110000111100001111111100000000000000000000000000000000\\
0010000100001100000011111111111100000000000000000000000000000000\\
0001000010000011111111110000111100000000000000000000000000000000\\
0000100001111111001100110011001100000000000000000000000000000000\\
0000011111010101010101010101010100000000000000000000000000000000\\
0000000000000000000000000000000011000011110000111111111100000000\\
0000000000000000000000000000000000110000001111111111110011000000\\
0000000000000000000000000000000000001111111111000011110000110000\\
0000000000000000000000000000000011111100110011001100110000001100\\
0000000000000000000000000000000011111111001100110000110000000011\\
1111111111111111111111111111111110101010101010101010101010101010\\
       \end{smallmatrix}\right)$

\item $q=3$, $n=3$, $k=3$, $w=[1, 2, 3]$:
       $\left(\begin{smallmatrix}
001\\
112\\
210\\
       \end{smallmatrix}\right)$

\item $q=3$, $n=6$, $k=3$, $w=[3, 4, 5]$:
       $\left(\begin{smallmatrix}
111101\\
121011\\
100122\\
       \end{smallmatrix}\right)$

\item $q=3$, $n=9$, $k=3$, $w=[5, 6, 7]$:
       $\left(\begin{smallmatrix}
011011001\\
110002111\\
100121202\\
       \end{smallmatrix}\right)$
\item $q=3$, $n=9$, $k=4$, $w=[3, 6, 9]$:
       $\left(\begin{smallmatrix}
100111110\\
010201211\\
221211121\\
112112221\\
       \end{smallmatrix}\right)$

\item $q=3$, $n=18$, $k=4$, $w=[9, 12, 15]$:
       $\left(\begin{smallmatrix}
111111110000001000\\
001112221111100100\\
120120120112210010\\
002000221120110001\\
       \end{smallmatrix}\right)$
\item $q=3$, $n=18$, $k=5$, $w=[9, 12, 15]$:
       $\left(\begin{smallmatrix}
111111110000010000\\
000011221111001000\\
011201010012100100\\
111122220012200010\\
012021121200000001\\
       \end{smallmatrix}\right)$
\item $q=3$, $n=18$, $k=6$, $w=[9, 12, 15]$:
       $\left(\begin{smallmatrix}
110011111100100000\\
001111122200010000\\
010201201111001000\\
112112212202000100\\
220001011221000010\\
122011022001000001\\
       \end{smallmatrix}\right)$

\item $q=3$, $n=27$, $k=4$, $w=[15, 18, 21]$:
       $\left(\begin{smallmatrix}
111111111111110000000001000\\
000011111222221111111000100\\
112200122001120011222110010\\
120212001120211202012120001\\
       \end{smallmatrix}\right)$
\item $q=3$, $n=27$, $k=5$, $w=[15, 18, 21]$:
       $\left(\begin{smallmatrix}
011011001111111111111110101\\
121000110120001122021011111\\
210101012002010120222102222\\
100110102211111221112112020\\
002201101200222211110001221\\
       \end{smallmatrix}\right)$
\item $q=3$, $n=27$, $k=6$, $w=[15, 18, 21]$:
       $\left(\begin{smallmatrix}
000111110100110011111111101\\
001210101001100122112121011\\
012200010011001212122221101\\
122000100110012110221221011\\
220002001101121100211222101\\
200021011010211002111212011\\]
       \end{smallmatrix}\right)$
\item $q=3$, $n=27$, $k=5$, $w=[9, 18, 27]$:
       $\left(\begin{smallmatrix}
111111110000000000000010000\\
000000001111111100000001000\\
001112220011122211111000100\\
120120121201201201122100010\\
121202011212020120112200001\\
       \end{smallmatrix}\right)$

\item $q=3$, $n=36$, $k=5$, $w=[21, 24, 27]$:
       $\left(\begin{smallmatrix}
1 1 1 0 0 1 0 0 1 1 1 1 0 1 1 1 1 1 1 1 1 0 1 1 1 1 1 0 1 1 1 0 1 0 1 1\\
0 0 1 1 1 2 1 1 0 0 1 2 1 1 0 2 2 2 2 2 1 1 2 0 2 2 1 1 1 0 2 1 1 0 1 0\\
2 2 0 0 1 0 0 2 2 1 1 0 2 2 1 2 1 2 1 1 2 2 1 2 1 0 0 2 1 2 2 1 0 1 1 0\\
0 2 2 2 0 1 1 0 0 2 1 2 2 0 2 1 2 1 2 1 1 1 0 1 2 1 2 2 0 1 2 1 0 1 0 1\\
1 0 0 2 2 0 2 1 2 0 1 1 1 2 2 0 2 2 1 2 2 2 2 2 0 2 1 2 2 1 0 0 1 1 0 1\\
       \end{smallmatrix}\right)$
\item $q=3$, $n=36$, $k=6$, $w=[21, 24, 27]$:
       $\left(\begin{smallmatrix}
1 0 1 0 1 1 1 1 1 1 0 0 0 0 1 0 1 1 0 0 1 1 0 1 0 0 1 0 1 1 1 0 1 1 0 1\\
1 1 0 1 0 2 0 2 1 1 1 0 1 0 0 1 0 1 1 0 0 1 1 0 1 0 0 1 0 2 2 1 0 2 1 0\\
2 1 1 0 2 0 0 0 2 1 1 1 1 1 0 0 2 0 0 1 0 0 1 1 0 1 0 0 1 0 1 2 1 0 2 1\\
0 2 1 1 0 1 2 0 0 2 1 1 0 1 2 0 0 2 0 0 1 0 0 1 1 0 1 1 0 2 0 1 2 2 0 2\\
1 0 2 1 2 0 2 1 0 0 2 1 2 0 2 2 0 0 1 0 0 1 0 0 1 1 0 2 1 0 1 0 1 1 2 0\\
0 1 0 2 2 1 2 1 1 0 0 2 0 2 0 2 1 0 0 1 0 0 1 0 0 1 1 1 2 2 0 1 0 0 1 2\\
       \end{smallmatrix}\right)$
\end{itemize}

\end{document}